\documentclass [10pt] {amsart}
\usepackage{latexsym}
\usepackage{amssymb}

\usepackage{amssymb}
\usepackage{amsmath}
\usepackage{amsthm}
\usepackage{color}
\usepackage{hyperref}
\usepackage{graphics}

\newcommand{\Bd}{\mathop{\mathrm{Bd}}}

\newtheorem{theorem}{Theorem}[section]

\newtheorem{proposition}[theorem]{Proposition}
\newtheorem{lemma}[theorem]{Lemma}
\newtheorem{corollary}[theorem]{Corollary}
\theoremstyle{definition}
\newtheorem{example1}[theorem]{Example}

\begin{document}

  \title[Computability of semicomputable manifolds]{Computability of semicomputable manifolds in computable topological spaces}
  \author{Zvonko
    Iljazovi\'{c}}
  \author{Igor Su\v{s}i\'{c}}
  
 \begin{abstract}
We study computable topological spaces and semicomputable and
computable sets in these spaces. In particular, we investigate
conditions under which  semicomputable sets are computable. We
prove that a semicomputable compact manifold $M$ is computable if
its boundary $\partial M$ is computable.  We also show how this
result combined with certain construction which compactifies a
semicomputable set  leads to the conclusion that some noncompact
semicomputable manifolds in computable metric spaces are
computable.
\end{abstract} 
  
    \maketitle

\section{Introduction}\label{introd}
A real number is computable if it can be effectively approximated
 by a rational number with arbitrary precision \cite{turing}. A tuple $(x_{1}
 ,\dots ,x_{n} )\in \mathbb{R}^{n}$ is computable if $x_{1} ,\dots
 ,x_{n} $ are computable numbers.  A compact subset of
$\mathbb{R}^{n} $ is computable if it can be effectively
approximated  by a finite set of points with rational coordinates
with arbitrary precision \cite{br-we}. Each nonempty computable
subset of $\mathbb{R}^{n} $ contains computable points, moreover
they are dense in it.

Suppose $f:\mathbb{R}^{n} \rightarrow \mathbb{R}$ is a computable
function (in the sense of \cite{per,we}) such that the set $f^{-1}
(\{0\})$ is compact. Does $f^{-1} (\{0\})$ have to be a computable
set?

 It is known that there exists a computable function
$f:\mathbb{R}\rightarrow \mathbb{R}$ which has zero-points and all
of them lie in $[0,1]$, but none of them is computable
\cite{specker}. So $f^{-1} (\{0\})$ is a nonempty compact set
which contains no computable point. In particular $f^{-1} (\{0\})$
is not computable, in fact we might say it is ``far away from
being computable''.

Hence for a function $f:\mathbb{R}^{n} \rightarrow \mathbb{R}$
such that $f^{-1} (\{0\})$ is a compact set the implication
\begin{equation}\label{fc-f0c}
f\mbox{ computable }\Rightarrow f^{-1} (\{0\}) \mbox{ computable}
\end{equation}
does not hold in general. The question is are there any additional
assumptions under which (\ref{fc-f0c}) holds. It turns out that
such assumptions exist and that certain topological properties of
the set $f^{-1} (\{0\})$ play an important role in this sense.

But before explaining what are these topological properties, we
will give another view to implication (\ref{fc-f0c}).

 A compact subset $S$ of
$\mathbb{R}^{n} $ is semicomputable if we can effectively
enumerate all rational open sets which cover $S$. It turns out
that a compact subset $S$ of $\mathbb{R}^{n} $ is semicomputable
if and only if $S=f^{-1} (\{0\})$ for some computable function
$f:\mathbb{R}^{n} \rightarrow \mathbb{R}$. Therefore closely
related to the question under which conditions (\ref{fc-f0c})
holds is the question under which conditions for $S\subseteq
\mathbb{R}^{n}$ the following implication holds:
\begin{equation}\label{intro-2}
S\mbox{ semicomputable } \Rightarrow S \mbox{ computable.}
\end{equation}
That (\ref{intro-2}) does not hold in general we conclude from the
fact that (\ref{fc-f0c}) does not hold in general. However
(\ref{intro-2}) does hold under some topological conditions on
$S$.

In order to see what is the role of topology in view of
(\ref{intro-2}), let us first observe the simple case when $S$ is
a line segment in $\mathbb{R}$. In \cite{miller} it is given an
example of a number $\gamma \in \mathbb{R}$ such that $[\gamma
,1]$ is a semicomputable, but not a computable subset of
$\mathbb{R}$.
On the other hand, if $a$ and $b$ are
computable numbers such that $a<b$, then $[a,b]$ is a computable
set. So (\ref{intro-2}) does not hold in general if $S$ is a line
segment in $\mathbb{R}$, but it does hold under additional
assumption that the endpoints of $S$ are computable.

The line segments and the arcs are the same in $\mathbb{R}$, but
in higher-dimensional Euclidean spaces the arcs are much more
general than the line segments. In view of the previous fact the
following question arises: does (\ref{intro-2}) holds if $S$ is an
arc in $\mathbb{R}^{n} $ with computable endpoints?

The answer to this question is not obvious. That the answer is
affirmative follows from the more general result of Miller
\cite{miller}: every semicomputable topological sphere in
$\mathbb{R}^{n} $ is computable and every semicomputable cell in
$\mathbb{R}^{n} $ with computable boundary sphere is computable.
Miller's pioneer work regarding conditions under which
(\ref{intro-2}) holds shows that topology has an important role in
view of these conditions.

That these results hold in a larger class of computable metric
spaces were shown in \cite{lmcs:cell}. The more general result was
later proved in \cite{lmcs:mnf}: (\ref{intro-2}) holds if $S$ is a
compact manifold with computable boundary (see also \cite{apal}).

Topological properties can force a semicomputable set $S$ to be
computable not just when $S$ is locally Euclidean. Chainable and
circularly chainable continua are generalizations of arcs and
topological circles and it is proved that (\ref{intro-2}) holds if
$S$ is a continuum chainable from $a$ to $b$, where $a$ and $b$
are computable points, or $S$ is a circularly chainable continuum
which is not chainable \cite{jucs, compintpoint}. Certain results when the complement of $S$ is disconnected can be found in \cite{jucs, tocs}.

The notions of semicomputable and computable set can be
generalized to noncompact sets and it turns out that
(\ref{intro-2}) does not hold in general if $S$ is a (noncompact)
1-manifold with computable boundary \cite{lmcs:1mnf}. However, it
is proved that (\ref{intro-2}) holds if $S$ is a 1-manifold with
computable boundary under additional condition that $S$ has
finitely many connected components. Certain conditions under which
(\ref{intro-2}) holds if $S$ is the graph of a function can be
found in \cite{brattka-pl}.


On the other hand, Kihara constructed in \cite{kihara}, as the
answer to a question in \cite{ziegler}, an  example of
a nonempty semicomputable compact set in the plane which is simply
connected (in fact, it is contractible) and which does not contain
any computable point. There also exists a semicomputable set of a
positive measure without a computable point \cite{tanaka}.

In Euclidean space a set is semicomputable if and only if it is
co-computably enumerable. That a set $S\subseteq \mathbb{R}^{n} $
is co-computably enumerable (co-c.e.) means that its complement
$\mathbb{R}^{n}\setminus S$  can be effectively covered by open
balls. For example the famous Mandelbrot set is co-c.e.\ (see
\cite{hert2}).

In this paper we put the investigation of conditions under which
(\ref{intro-2}) holds into the more general ambient space:
computable topological space. The notion of a computable
topological space is not new, for example see \cite{jucs-we,jucs-we2}. We
will use the notion of a computable topological space which corresponds to the notion of a SCT$_{2}$ space from \cite{jucs-we2} and we will investigate some of its aspects. We will
see how to each computable metric space can be naturally
associated a computable topological space and how the notions of a
semicomputable and a computable set can be easily extend to
computable topological spaces.

The central part of this paper will be the proof of the main
result, i.e.\ the proof of the fact that (\ref{intro-2}) holds in
any computable topological space if $S$ is a compact manifold with
computable boundary. This will be a generalization of the result
from \cite{lmcs:mnf}. Although we will rely on certain ideas from
\cite{lmcs:mnf}, the main challenge will be to adopt ideas and
techniques from \cite{lmcs:mnf}, which  depend on the metric $d$
in a computable metric space $(X,d,\alpha )$, to an ambient in
which we do not have any metric. For example, the notion that a
set $S$ is computable up to a set $T$, which means that for each
$k\in \mathbb{N}$ we can effectively find finitely many points
$x_{0} ,\dots ,x_{n} $ such that each point of $S$ is
$2^{-k}$-close to some $x_{i} $ and each $x_{i} $ is
$2^{-k}$-close to some point of $T$, was essential in
\cite{lmcs:mnf} and it is not obvious how to transfer it in a
nonmetric setting. Another example is the notion of the formal
diameter of a set in a computable metric space which is a
computable analogue of the diameter of a set in a metric space and
which clearly does not make sense in a (computable) topological
space.

The generalization of the result for manifolds from
\cite{lmcs:mnf} to computable topological spaces does not only
show that a metric in this context is not really important, but it
also provides a possible tool for dealing with the problem of
computability of a semicomputable noncompact set $S$ in a
computable metric space $(X,d,\alpha )$. Namely, using a
construction similar to the one-point compactification, we can
assign to $(X,d,\alpha )$ a computable topological space $T$ in
such a way that, under this construction, $S$ maps to a compact
set $S'$ in $T$ and such that the computability of $S'$ in $T$
implies the computability of $S$ in $(X,d,\alpha )$. We will see
how this gives that a semicomputable set in a computable metric
space homeomorphic to  $\mathbb{R}^{n} $ (for some $n$) must be
computable.

It should be mentioned  that the uniform version of the result from \cite{lmcs:mnf}
 does not hold in general: there exists a
 sequence $(S_{i} )$ of topological circles in $\mathbb{R}^{2}$
 such that $S_{i} $ is uniformly semi-computable, but
 not uniformly computable (Example 7 in \cite{jucs}).

Here is how the paper is organized. In Section \ref{sect-prelim} we state some basic definitions and facts. In Section \ref{sect-cts} we study the notion of a computable topological space and in Section \ref{sect-escs}  we examine effective separations of compact sets in computable topological spaces. In Section \ref{sect-lce} we introduce the notion of local computable enumerability of a set as a preparation for Sections \ref{sect-scm} and \ref{sect-scm-wb} in which we prove our main result: a semicomputable manifold in a computable topological space is computable if its boundary is semicomputable. In Section \ref{sect-cas} we reduce the problem of computability of \emph{noncompact} semicomputable sets in a computable metric space to the problem of computability of \emph{compact} semicomputable sets in a computable topological space.

\section{Computable metric spaces and preliminaries} \label{sect-prelim}
In this section we give some basic facts about computable metric
spaces and some other preliminary facts. See \cite{per, we,
turing, we2, br-we, br-pr, jucs}.
\subsection{Computable functions $\mathbb{N}^{k} \rightarrow
\mathbb{R}$}

 Let $k\in \mathbb{N}$, $k\geq 1$. A function $f:\mathbb{N}^{k}
\rightarrow \mathbb{Q}$ is said to be \textbf{computable} if there
exists computable (i.e.\ recursive) functions $f_{0},f_{1}
,f_{2}:\mathbb{N}^{k} \rightarrow \mathbb{N}$ such that
$$f(x)=(-1)^{f_{0} (x)}\frac{f_{1} (x)}{f_{2} (x)+1}$$
for each $x\in \mathbb{N}^{k} $. A function $f:\mathbb{N}^{k}
\rightarrow \mathbb{R}$ is said to be \textbf{computable} if there
exists a computable function $F:\mathbb{N}^{k+1}\rightarrow
\mathbb{Q}$ such that $$|f(x)-F(x,i)|<2^{-i}$$ for each $x\in
\mathbb{N}^{k} $ and each $i\in \mathbb{N}$. Of course, a function
$\mathbb{N}^{k} \rightarrow \mathbb{R}^{n} $ or $\mathbb{N}^{k}
\rightarrow \mathbb{Q}^{n} $, where $n\in \mathbb{N}$, $n\geq 1$,
will be called \textbf{computable} if its component functions are
computable.

Some elementary properties of computable functions $\mathbb{N}^{k}
\rightarrow \mathbb{R}$ are stated in the following proposition.
\begin{proposition} \label{NuR}
\begin{enumerate}

\item[(i)] If $f,g:\mathbb{N}^{k} \rightarrow \mathbb{R}$ are
computable, then $f+g,f-g, f\cdot g:\mathbb{N}^{k} \rightarrow
\mathbb{R}$ are computable.

\item[(ii)] If $f:\mathbb{N}^{k} \rightarrow \mathbb{R}$ and
$F:\mathbb{N}^{k+1} \rightarrow \mathbb{R}$ are functions such
that $F$ is computable and $|f(x)-F(x,i)|<2^{-i}$ for each $x\in
\mathbb{N}^{k}$ and $i\in \mathbb{N}$, then $f$ is computable.

\item[(iii)] If $f,g:\mathbb{N}^{k} \rightarrow \mathbb{R}$ are
computable functions, then the set $\{x\in \mathbb{N}^{k} \mid
f(x)>g(x)\}$ is computably enumerable. \qed

\end{enumerate}
\end{proposition}
\subsection{Computable metric spaces}
A triple $(X,d,\alpha )$ is said to be a \textbf{computable metric
space} if $(X,d)$ is a metric space and $\alpha =(\alpha _{i} )$
is a sequence whose range is dense in $(X,d)$ and such that the
function $\mathbb{N}^{2}\rightarrow \mathbb{R}$, $$(i,j)\mapsto
d(\alpha _{i} ,\alpha _{j} )$$ is computable (see \cite{brattka-pl,br-pr,we2,bsl}). For example, if
$n\geq 1$ and $d$ is the Euclidean metric on $\mathbb{R}^{n} $,
then for any computable function $\alpha :\mathbb{N}\rightarrow
\mathbb{R}^{n}$ whose range is dense in $\mathbb{R}^{n} $ we have
that $(\mathbb{R}^{n} ,d,\alpha )$ is a computable metric space.
(Such a function $\alpha $ certainly exists: we can take a
computable surjection $\alpha :\mathbb{N}\rightarrow
\mathbb{Q}^{n} $.)

Let us recall the notion of the Hausdorff distance. If $(X,d)$ is
a metric space and $S$ and $T$ nonempty compact sets in this
space, we define their \textbf{Hausdorff distance} $d_{H} (S,T)$
by
$$d_{H} (S,T)=\inf\{\varepsilon >0\mid S\approx_{\varepsilon
}T\},$$ where $S\approx_{\varepsilon }T$ means that for each $x\in
S$ there exists $y\in T$ such that $d(x,y)<\varepsilon $ and for
each $y\in T$ there exists $x\in S$ such that $d(y,x)<\varepsilon
$.

Let $(X,d,\alpha )$ be a computable metric space and $x \in X$.
Then for each $k\in \mathbb{N}$ there exists $i\in \mathbb{N}$
such that $d(x,\alpha _{i} )<2^{-k}$. We say that $x $ is a
\textbf{computable point} in $(X,d,\alpha )$ if there exists a
computable function $f:\mathbb{N}\rightarrow \mathbb{N}$ such that
\begin{equation}\label{comp-point}
d(x ,\alpha _{f(k)})<2^{-k}
\end{equation}
for each $k\in \mathbb{N}$.

Suppose now $S$ is a nonempty compact set in $(X,d)$. Then the
density of $\alpha $ implies that for each $k\in \mathbb{N}$ there
exists a nonempty finite subset $A$ of $\{\alpha _{i} \mid i\in
\mathbb{N}\}$ such that $d_{H} (S,A)<2^{-k}$. This fact naturally
leads to a definition of a computable (compact) set.

First, we will fix some effective enumeration of all nonempty finite subsets of $\mathbb{N}$. To do this, we will use the following notion.

Let $k,n\in \mathbb{N}$, $k,n\geq 1$, and $\Phi :\mathbb{N}^{k}
\rightarrow \mathcal{P}(\mathbb{N}^{n} )$, where
$\mathcal{P}(\mathbb{N}^{n} )$ denotes the power set of
$\mathbb{N}^{n}$. We say that $\Phi$ is \textbf{computably finite
valued} (c.f.v.) if $$\{(x,y)\in \mathbb{N}^{k}\times
\mathbb{N}^{n} \mid y\in \Phi (x)\}$$ is a computable subset of
$\mathbb{N}^{k+n}$ and there exists a computable function $\varphi
:\mathbb{N}^{k} \rightarrow \mathbb{N}$ such that $\Phi
(x)\subseteq \{0,\dots , \varphi (x)\}^{n} $ for each $x\in
\mathbb{N}^{k}$.

From now on, let $\mathbb{N}\rightarrow \mathcal{P}(\mathbb{N})$,
\begin{equation}\label{j-[j]}
j\mapsto [j],    
\end{equation}
be some fixed c.f.v.\ function whose image is the set of all nonempty finite subsets of $\mathbb{N}$ (such a function certainly exists). Hence, $([j])_{j\in \mathbb{N}}$ is an effective enumeration of all nonempty finite subsets of $\mathbb{N}$.

Let $(X,d,\alpha )$ be a computable metric space. For $j\in
\mathbb{N}$ we define $$\Lambda _{j} =\{\alpha _{i} \mid i\in
[j]\}.$$ Let $S$ be a compact set in $(X,d)$. We say that $S$ is a
\textbf{computable set} in $(X,d,\alpha )$ if $S=\emptyset $ or
there exists a computable function $f:\mathbb{N}\rightarrow
\mathbb{N}$ such that $$d_{H} (S,\Lambda _{f(k)})<2^{-k}$$ for
each $k\in \mathbb{N}$. It is not hard to conclude that this
definition does not depend on the choice of the function $([j])_{j\in \mathbb{N}}$ (see Proposition \ref{cfv}).

If $(X,d,\alpha )$ is a computable metric space, $i\in \mathbb{N}$
and $r$ a positive rational number, then we say that $B(\alpha
_{i} ,r)$ is a \textbf{rational open ball} in $(X,d,\alpha )$.
Here, for $x\in X$ and $r>0$, we denote by $B(x,r)$ the open ball
of radius $r$ centered at $x$, i.e.\ $B(x,r)=\{y\in X\mid
d(x,y)<r\}$. A finite union of rational open balls will be called
a \textbf{rational open set} in $(X,d,\alpha )$.

Let $q:\mathbb{N}\rightarrow \mathbb{Q}$ be some fixed computable
function whose image is the set of all positive rational numbers
and let $\tau_{1} ,\tau_{2} :\mathbb{N}\rightarrow \mathbb{N}$ be
some fixed computable functions such that $\{(\tau _{1} (i), \tau
_{2} (i))\mid i\in \mathbb{N}\}=\mathbb{N}^{2}.$

Let $(X,d,\alpha )$ be a computable metric space.  Let $(\lambda
_{i} )_{i\in \mathbb{N}}$ be the sequence of points in $X$ defined
by $\lambda _{i} =\alpha _{\tau _{1} (i)}$ and let $(\rho _{i}
)_{i\in \mathbb{N}}$ be the sequence of rational numbers defined
by $\rho _{i} =q_{\tau _{2} (i)}$. For $i\in \mathbb{N}$ we define
\begin{equation}\label{defIi}
I_{i}=B(\lambda_{i}   ,\rho_{i}  ).
\end{equation}
 Note that  $\{I_{i} \mid
i\in \mathbb{N}\}$ is the family of all open rational balls in
$(X,d,\alpha )$. For $j\in \mathbb{N}$ we define
$$J_{j}=\bigcup_{i\in [j]}I_{i}.$$
Clearly  $\{J_{i} \mid i\in \mathbb{N}\}$ is the family of all
rational open sets in $(X,d,\alpha )$.

A closed set $S$ in $(X,d)$ is said to be \textbf{computably
enumerable} (c.e.) in $(X,d,\alpha )$ if the set $\{i\in
\mathbb{N}\mid I_{i} \cap S\neq \emptyset \}$ is c.e. A compact
set $S$ in $(X,d)$ is said to be \textbf{semicomputable} in
$(X,d,\alpha )$ if the set $\{j\in \mathbb{N}\mid S\subseteq J_{j}
\}$ is c.e. It is not hard to see that these definitions do not
depend on the choice of the functions $q$, $\tau _{1} $, $\tau
_{2} $ and $([j])_{j\in \mathbb{N}}$.

We have the following characterization of a computable set
(Proposition 2.6 in \cite{lmcs:mnf}):
\begin{equation}\label{compset-ce-semi}
S\mbox{ computable in }(X,d,\alpha ) \Leftrightarrow S\mbox{ c.e.\
and semicomputable in }(X,d,\alpha ).
\end{equation}

Let $(X,d,\alpha )$ be a computable metric space and $x\in X$. If
$x$ is a computable point in $(X,d,\alpha )$, then there exists a
computable function $f:\mathbb{N}\rightarrow \mathbb{N}$ such that
(\ref{comp-point}) holds. Since for all $a,b,c\in X$ we have
$|d(a,c)-d(b,c)|\leq d(a,b)$, for all $i,k\in \mathbb{N}$ we have
$$|d(x,\alpha _{i}) -d(\alpha _{f(k)},\alpha _{i} )|\leq d(x,\alpha
_{f(k)})<2^{-k}$$ and it follows from Proposition \ref{NuR}(ii)
that the function $\mathbb{N}\rightarrow \mathbb{R}$, $i\mapsto
d(x,\alpha _{i} )$ is computable. Thus the function
$\mathbb{N}\rightarrow \mathbb{R}$, $i\mapsto d(x,\lambda  _{i} )$
is also computable. For $i\in \mathbb{N}$ we have
$$x\in I_{i} \Leftrightarrow d(x,\lambda _{i} )<\rho _{i} $$ and
Proposition \ref{NuR}(iii) implies that the set $\{i\in
\mathbb{N}\mid x\in I_{i} \}$ is c.e.

Conversely, if the set $\{i\in \mathbb{N}\mid x\in I_{i} \}$ is
c.e., then the set $\Omega =\{(k,i)\in \mathbb{N}^{2}\mid x\in
I_{i}$ and $\rho _{i} <2^{-k} \}$ is also c.e.\ and since for each
$k\in \mathbb{N}$ there exists $i\in \mathbb{N}$ such that
$(k,i)\in \Omega $, there exists
a computable function $f:\mathbb{N}\rightarrow \mathbb{N}$ such
that $(k,f(k))\in \Omega $ for each $k\in \mathbb{N}$. So
$d(x,\lambda _{f(k)})<2^{-k}$ for each $k\in \mathbb{N}$ and it
follows that $x$ is a computable point. We have the following
conclusion:
\begin{equation}\label{comp-point-c.e.}
x\mbox{ computable point in }(X,d,\alpha )\Leftrightarrow \{i\in
\mathbb{N}\mid x\in I_{i} \}\mbox{ c.e.\ set}.
\end{equation}

Let $(X,d,\alpha )$ be a computable metric space and let $S\subseteq X$. We say that $S$ is a \textbf{co-computably enumerable} (co-c.e.) set in $(X,d,\alpha )$ if there exists a c.e.\ set $A\subseteq \mathbb{N}$ such that $$X\setminus S=\bigcup_{i\in A}I_{i} .$$ We say that $S$ is a \textbf{computable closed} set if $S$ is both c.e.\ and co-c.e.

Each computable set is a computable closed set \cite{lmcs:mnf}. Conversely, a computable closed set need not be computable even if it is compact. However, if $(X,d,\alpha )$ has the effective covering property (for the definition see \cite{br-pr}) and compact closed balls, then for compact sets the notions ``computable'' and ``computable closed'' coincide (Proposition 3.6 in \cite{lmcs:1mnf}).
This in particular holds in the previously described computable metric space $(\mathbb{R}^{n},d,\alpha )$.

\subsection{Formal properties}
Let $(X,d)$ be a metric space, $x,y\in X$ and $r,s>0$. If
$d(x,y)\geq r+s$, then $B(x,r)\cap B(y,s)=\emptyset $. Conversely,
if $B(x,r)\cap B(y,s)=\emptyset $, then inequality $d(x,y)\geq
r+s$ holds if $(X,d)$ is Euclidean space, but it does not hold in
general (for example, if $d$ is the discrete metric).
Nevertheless, we will use this inequality (actually, the strong
inequality) to introduce certain relation of formal disjointness
between rational open balls $I_{i} $ and $I_{j} $ (actually
between the numbers $i$ and $j$) in a computable metric space.

Similarly, if $d(x,y)+s\leq r$, then $B(y,s)\subseteq B(x,r)$, but
the converse of this statement does not hold in general. Although
this inequality does not characterize the fact that
$B(y,s)\subseteq B(x,r)$, it will be useful for us in computable
metric spaces to introduce certain notion of formal inclusion.

 Let $(X,d,\alpha )$ be a computable metric space.
 Let $i,j\in
\mathbb{N}$. (Recall the definition
 (\ref{defIi}).) We say that $I_{i} $ and $I_{j} $ are
\textbf{formally disjoint} and write $I_{i} \diamond I_{j} $ if
$$d(\lambda _{i} ,\lambda _{j} )>\rho _{i} +\rho _{j} .$$ We say
that $I_{i} $ is \textbf{formally contained} in $I_{j} $ and write
$I_{i} \subseteq _{F}I_{j} $ if
$$d(\lambda _{i} ,\lambda _{j} )+\rho _{i}<\rho _{j} .$$
The main properties of these two relations are stated in the next
proposition.
\begin{proposition} \label{svojstva-diamond-F}
Let $(X,d,\alpha )$ be a computable metric space. Then the sets
\begin{equation}\label{sets-1}
\{(i,j)\in \mathbb{N}^{2}\mid I_{i} \diamond I_{j} \}\mbox{ and
}\{(i,j)\in \mathbb{N}^{2}\mid I_{i} \subseteq _{F}I_{j} \}
\end{equation}
 are
c.e. Furthermore,  the following
holds:
\begin{enumerate}
\item if $i,j\in \mathbb{N}$ are such that $I_{i} \diamond I_{j}$,
then $I_{i} \cap I_{j} =\emptyset $;

\item if $i,j\in \mathbb{N}$ are such that $I_{i} \subseteq _{F}
I_{j}$, then $I_{i}\subseteq I_{j}$;

\item if $x,y\in X$ are such that $x\neq y$, then there exist
$i,j\in \mathbb{N}$ such that $x\in I_{i} $, $y\in I_{j} $ and
$I_{i} \diamond I_{j} $;

\item if $i,j\in \mathbb{N}$ and $x\in I_{i} \cap  I_{j} $, then
there exists $k\in \mathbb{N}$ such that $x\in I_{k} $, $I_{k}
\subseteq _{F}I_{i} $ and $I_{k} \subseteq _{F}I_{j} $; moreover,
if $A\subseteq \{\alpha _{i} \mid i\in \mathbb{N}\}$ is a dense
set in $(X,d)$, $k$ can be chosen so that $\lambda _{k} \in A$;

\item if $i,j\in \mathbb{N}$ are such that $I_{i} \diamond I_{j}$,
then $I_{j} \diamond I_{i} $; \item if $i,j,k\in \mathbb{N}$ are
such that $I_{i} \subseteq _{F} I_{j}$ and $I_{j}\subseteq
_{F}I_{k} $, then $I_{i}\subseteq_{F} I_{k}$; \item if $i,j,k\in
\mathbb{N}$ are such that $I_{k} \subseteq _{F}I_{i} $ and
$I_{i}\diamond I_{j} $, then $I_{k} \diamond I_{j} $.
\end{enumerate}
\end{proposition}
\begin{proof}
It follows from the definition of $I_{i} \diamond I_{j} $ and $I_{i} \subseteq _{F}I_{j} $
and Proposition \ref{NuR} that the sets in (\ref{sets-1}) are c.e. Furthermore,
claims (1) and (2) obviously hold.

Let us prove (3). Suppose  $x,y\in X$, $x\neq y$. Let $r =
\frac{d(x,y)}{4}$. Choose  $k \in \mathbb{N}$ so that  $q_k < r$
and
 $u$, $v \in \mathbb{N}$ so that
$$x \in B(\alpha_u, q_k) \text{~ and ~} y \in B(\alpha_v, q_k).$$
There exist $i$, $j \in \mathbb{N}$ such that $(u,k) = \big(
\tau_1(i), \tau_2(i) \big)$ and $(v,k) = \big( \tau_1(j),
\tau_2(j) \big)$ and therefore $(\alpha_u, q_k)=(\lambda _{i}
,\rho _{i} )$ and $(\alpha_v, q_k)=(\lambda _{j} ,\rho _{j} )$. So
$$
x \in I_i \text{~ and ~} y \in I_j.
$$
We claim that $I_i\diamond I_j$. Suppose the opposite. Then
$d(\lambda _{i} ,\lambda _{j} )\leq \rho _{i} +\rho _{j} $, i.e.\
$d(\alpha _{u},\alpha _{v})\leq 2q_{k} $. We have
$$ d(x,y) \leq d(x,\alpha_u) + d(\alpha_u,
\alpha_v) + d(\alpha_v,y) < q_k + 2q_k + q_k= 4q_k < 4r =
d(x,y),$$ i.e.\ $d(x,y)<d(x,y)$, a contradiction. Hence
$I_i\diamond I_j$.

Let us prove (4). Suppose $i,j\in \mathbb{N}$ and $x\in I_{i} \cap
I_{j} $. Then $d(x,\lambda _{i} )<\rho _{i}$ and $d(x,\lambda _{j}
)<\rho _{j} $. Choose $v\in \mathbb{N}$ such that
$$d(x,\lambda _{i} ) + 2q_{v} <
\rho _{i}  \mbox{ and }d(x,\lambda _{j} ) + 2q_{v} < \rho _{j} .$$
Choose $u \in \mathbb{N}$ so that $d(x,\alpha_u) < q_{v}$ and
$\alpha _{u}\in A$.

Let $k \in \mathbb{N}$ be such that $(\alpha _{u},q_{v}) =
(\lambda _{k} ,\rho _{k} )$. Then $x\in I_k$. Furthermore,
$$
d(\lambda _{k} , \lambda _{i} ) + \rho _{k}  = d(\alpha_u, \lambda
_{i} ) + q_v\leq d(\alpha_u, x) + d(x, \lambda _{i} ) + q_v< d(x,
\lambda _{i} ) + 2q_{v}< \rho _{i} .$$ Hence $I_k \subseteq_F
I_i$. In the same way we get $I_k \subseteq_F I_j$.

Claim (5) is obvious. It is straightforward to check that (6)
holds.

We now prove (7). Suppose $I_{k} \subseteq _{F}I_{i} $ and $I_{i}
\diamond I_{j} $.  Since $I_k \subseteq_F I_i$, we have
\begin{equation}
\label{ijk-1} \rho _{k}  < \rho _{i}  - d(\lambda _{i} , \lambda
_{k} ).
\end{equation}
We also have $ \rho _{i}  + \rho _{j}  < d(\lambda _{i} , \lambda
_{j}  ) \leq d(\lambda _{i} , \lambda _{k} ) + d(\lambda _{k} ,
\lambda _{j} )$, so $ \rho _{i}  + \rho _{j}  < d(\lambda _{i} ,
\lambda _{k} ) + d(\lambda _{k} , \lambda _{j} )$ and therefore
\begin{equation}
\label{ijk-2} \rho _{i} -d(\lambda _{i} , \lambda _{k} ) <   -\rho
_{j}+ d(\lambda _{k} , \lambda _{j} ).
\end{equation}
It follows from (\ref{ijk-1}) and (\ref{ijk-2}) that $\rho _{k}
<-\rho _{j}+ d(\lambda _{k} , \lambda _{j} )$, hence $I_k\diamond
I_j$.
\end{proof}
\subsection{Final remarks}  \label{subs-fr}
The following properties of c.f.v.\ functions will be useful.
\begin{proposition} \label{cfv}
\begin{enumerate}
\item[(1)] If $\Phi ,\Psi:\mathbb{N}^{k} \rightarrow
\mathcal{P}(\mathbb{N}^{n} )$ are c.f.v.\ functions, then the
function $\mathbb{N}^{k} \rightarrow \mathcal{P}(\mathbb{N}^{n}
)$, $x\mapsto \Phi (x)\cup \Psi (x)$ is  c.f.v.

\item[(2)] If $\Phi: \mathbb{N}^{k} \rightarrow
\mathcal{P}(\mathbb{N}^{n} )$ and $\Psi :\mathbb{N}^{k}
\rightarrow \mathcal{P}(\mathbb{N}^{l} )$ are c.f.v.\ functions,
then the function $\mathbb{N}^{k} \rightarrow
\mathcal{P}(\mathbb{N}^{n+l} )$, $x\mapsto \Phi (x)\times  \Psi
(x)$ is  c.f.v.

\item[(3)] If $\Phi ,\Psi:\mathbb{N}^{k} \rightarrow
\mathcal{P}(\mathbb{N}^{n} )$ are c.f.v.\ functions, then the sets
$\{x\in \mathbb{N}^{k} \mid \Phi (x)=\Psi (x)\}$ and $\{x\in
\mathbb{N}^{k} \mid \Phi (x)\subseteq \Psi (x)\}$ are computable.

\item[(4)] Let $\Phi :\mathbb{N}^{k} \rightarrow
\mathcal{P}(\mathbb{N}^{n} )$ and $\Psi :\mathbb{N}^{n}\rightarrow
\mathcal{P}(\mathbb{N}^{m} )$ be c.f.v.\ functions. Let $\Lambda
:\mathbb{N}^{k} \rightarrow \mathcal{P}(\mathbb{N}^{m} )$ be
defined by $$\Lambda (x)=\bigcup_{z\in \Phi (x)}\Psi  (z),$$ $x\in
\mathbb{N}^{k} $. Then $\Lambda  $ is a c.f.v.\ function.

\item[(5)] Let $\Phi :\mathbb{N}^{k} \rightarrow
\mathcal{P}(\mathbb{N}^{n} )$ be c.f.v$.$ and let $T\subseteq
\mathbb{N}^{n} $ be c.e. Then the set $S=\{x\in \mathbb{N}^{k}
\mid \Phi (x)\subseteq T\}$ is c.e. \qed
\end{enumerate}
\end{proposition}

Let $\sigma :\mathbb{N}^{2}\rightarrow \mathbb{N}$
and $\eta:\mathbb{N}\rightarrow \mathbb{N}$ be some fixed computable
functions with the following property: $\{(\sigma (j,0),\dots
,\sigma (j,\eta(j)))\mid j\in \mathbb{N}\}$ is the set of all
finite nonempty sequences in $\mathbb{N}$. We use the following
notation: $(j)_{i}$ instead of $\sigma (j,i)$ and $\overline{j}$
instead of $\eta(j).$  Hence $$\{((j)_{0} ,\dots
,(j)_{\overline{j}})\mid j\in \mathbb{N}\}$$ is the set of all
finite nonempty sequences in $\mathbb{N}.$ 

It follows from Proposition \ref{cfv}(4) that the function $\mathbb{N}\rightarrow \mathcal{P}(\mathbb{N})$,
$j\mapsto \{(j)_{i} \mid 0\leq i\leq \overline{j}\}$ is c.f.v. Clearly, the image of this function is the set of all nonempty finite subsets of $\mathbb{N}$. This means that we can take this function as an effective enumeration introduced by (\ref{j-[j]}) and it will suitable for us to do so. Therefore, we assume that   
\begin{equation}\label{p2-eq}
[j]=\{(j)_{i} \mid 0\leq i\leq \overline{j}\}.
\end{equation}
for each $j\in \mathbb{N}$.

\section{Computable topological spaces} \label{sect-cts}
Proposition \ref{svojstva-diamond-F} is a motivation for the next
definition.

Let $(X,\mathcal{T})$ be a topological space and let $(I_{i} )$ be
a sequences in $\mathcal{T}$ such that $\{I_{i} \mid i\in
\mathbb{N}\}$ is a basis for the topology $\mathcal{T}$. A triple
$(X,\mathcal{T},(I_{i} ))$ is said to be a \textbf{computable
topological space} (see the definition of a SCT$_{2}$ space in \cite{jucs-we2}) if there exist  c.e.\ subsets $\mathcal{C}$ and
$\mathcal{D}$ of $\mathbb{N}^{2}$ with the following properties:
\begin{enumerate}
\item if $i,j\in \mathbb{N}$ are such that $(i,j)\in \mathcal{D}$,
then $I_{i} \cap I_{j} =\emptyset $;

\item if $i,j\in \mathbb{N}$ are such that $(i,j)\in \mathcal{C}$,
then $I_{i}\subseteq I_{j}$;

\item if $x,y\in X$ are such that $x\neq y$, then there exist
$i,j\in \mathbb{N}$ such that $x\in I_{i} $, $y\in I_{j} $ and
$(i,j)\in \mathcal{D} $;

\item if $i,j\in \mathbb{N}$ and $x\in I_{i} \cap I_{j} $, then
there exists $k\in \mathbb{N}$ such that $x\in I_{k} $, $(k,i)\in
\mathcal{C} $ and $(k,j)\in \mathcal{C} $.
\end{enumerate}
In this case we say that $\mathcal{C}$ and $\mathcal{D}$ are
\textbf{characteristic relations} for $(X,\mathcal{T},(I_{i} ))$.

Note the following: if $(X,\mathcal{T},(I_{i} ))$  is a computable
topological space, then $(X,\mathcal{T})$ is a second countable
Hausdorff space.

Let $(X,d,\alpha )$ be a computable metric space. Let
$\mathcal{T}_{d} $ denote the topology induced by $d$, i.e.\ the
set of all open sets in $(X,d)$. Let, for $i\in \mathbb{N}$, the
set $I_{i} $ be defined by (\ref{defIi}) (for fixed functions $q$,
$\tau _{1} $ and $\tau _{2} $). Then $\{I_{i} \mid i\in
\mathbb{N}\}$ is a basis for the topology $\mathcal{T}_{d}$ and by
Proposition \ref{svojstva-diamond-F} $(X,\mathcal{T}_{d},(I_{i}
))$ is a computable topological space; characteristic relations
are $\{(i,j)\in \mathbb{N}^{2}\mid I_{i} \subseteq _{F}I_{j} \}$
and $\{(i,j)\in \mathbb{N}^{2}\mid I_{i} \diamond I_{j} \}$. We
say that $(X,\mathcal{T}_{d},(I_{i} ))$ is the computable
topological space associated to $(X,d,\alpha )$.

Let $(X,\mathcal{T},(I_{i} ))$ be a computable topological space.
Let $x\in X$. We say that $x$ is a \textbf{computable point} in
$(X,\mathcal{T},(I_{i} ))$ if the set $\{i\in \mathbb{N}\mid x\in
I_{i} \}$ is c.e.

A closed set $S$ in $(X,\mathcal{T})$ is said to be
\textbf{computably enumerable} in $(X,\mathcal{T},(I_{i} ))$ if
$\{i\in \mathbb{N}\mid I_{i} \cap S\neq\emptyset \}$ is a c.e.\
set.

If $(X,\mathcal{T},(I_{i} ))$ is a computable topological space,
then for $j\in \mathbb{N}$ we define $J_{j} $ by $$J_{j}
=\bigcup_{i\in [j]}I_{i} .$$

Let $(X,\mathcal{T},(I_{i} ))$ be a computable topological space
and let $S$ be a compact set in $(X,\mathcal{T})$. We say that $S$
is a \textbf{semicomputable set} in $(X,\mathcal{T},(I_{i} ))$ if
$\{j\in \mathbb{N}\mid S\subseteq J_{j} \}$ is a c.e.\ set. We say
that $S$ is a \textbf{computable set} in $(X,\mathcal{T},(I_{i}
))$ if $S$ is computably enumerable and semicomputable in
$(X,\mathcal{T},(I_{i} ))$. These definitions are easily seen to
be independent on the choice of the function $([j])_{j\in \mathbb{N}}$.

\begin{proposition}\label{cms-cts}
Let $(X,d,\alpha )$ be a computable metric space and let
$(X,\mathcal{T}_{d},(I_{i} ))$ be the associated computable
topological space. Let $x\in X$ and $S\subseteq X$. The the
following equivalences hold:
\begin{enumerate}
\item[(i)] $x$ computable point in $(X,d,\alpha )$
$\Leftrightarrow$ $x$ computable point in
$(X,\mathcal{T}_{d},(I_{i} ))$;

\item[(ii)] $S$ c.e.\ set in $(X,d,\alpha )$ $\Leftrightarrow$ $S$
c.e.\ set in $(X,\mathcal{T}_{d},(I_{i} ))$;

\item[(iii)] $S$ semicomputable set in $(X,d,\alpha )$
$\Leftrightarrow$ $S$ semicomputable set in
$(X,\mathcal{T}_{d},(I_{i} ))$;

\item[(iii)] $S$ computable set in $(X,d,\alpha )$
$\Leftrightarrow$ $S$ computable set in $(X,\mathcal{T}_{d},(I_{i}
))$.
\end{enumerate}
\end{proposition}
\begin{proof}
This follows from (\ref{comp-point-c.e.}) and
(\ref{compset-ce-semi}).
\end{proof}

In this paper we prove that in any computable topological space
$(X,\mathcal{T},(I_{i} ))$ the implication $$S\mbox{
semicomputable }\Rightarrow S\mbox{ computable}$$ holds if $S$ is,
as a subspace of $(X,\mathcal{T})$, a manifold whose boundary is
computable. By Proposition \ref{cms-cts} this is a generalization
of the result from \cite{lmcs:mnf} for semicomputable manifolds in
computable metric spaces.

Regarding the definition of a computable topological space, the
natural question is this: if $(X,\mathcal{T},(I_{i} ))$ is a
computable topological space, do there exist $d$ and $\alpha $
such that $(X,d,\alpha )$ is a computable metric space whose
associated computable topological space is $(X,\mathcal{T},(I_{i}
))$? In the following example we get that the answer is negative:
$(X,\mathcal{T})$ need not be metrizable, moreover it need not be
even regular (recall that $(X,\mathcal{T})$ is always second
countable Hausdorff). The example is motivated by a classical
example of a Hausdorff space which is not regular (see
\cite{ch-vox}).

\begin{example1}
Let $c\in \mathbb{R}\setminus \mathbb{Q}$ be a computable number.
Let $\beta :\mathbb{N}\rightarrow \mathbb{Q}$ be a computable
surjection and let $\gamma :\mathbb{N}\rightarrow \mathbb{R}$ be
defined by $\gamma (i)=c+\beta (i)$. Then $\gamma $ is a
computable function.

Let $\alpha :\mathbb{N}\rightarrow \mathbb{R}$ be defined by
$$\alpha_{i} =\left\{\begin{tabular}{l}
             $\beta _{\tau _{1} (i)}$ if $\tau _{2} (i)\in  2 \mathbb{N}$,  \\
             $\gamma _{\tau _{1}(i)}$ if $\tau _{2} (i)\notin  2
             \mathbb{N}$.
              \end{tabular} \right.$$
Then $\alpha $ is a computable function and $\{\alpha _{i} \mid
i\in \mathbb{N}\}=\mathbb{Q}\cup (c+ \mathbb{Q}).$

Let $X=\mathbb{Q}\cup (c+ \mathbb{Q})$ and let $d$ be the
Euclidean metric on $X$. Then $(X,d,\alpha )$ is a computable
metric space. Let the sequences $(\lambda _{i} )$, $(\rho _{i} )$
and $(I_{i} )$ for this computable metric space be defined in the
standard way. For $i\in \mathbb{N}$ we define $$B_{i} =(I_{i} \cap
\mathbb{Q})\cup \{\lambda _{i} \}.$$ Let $\mathcal{D}=\{(i,j)\in
\mathbb{N}^{2}\mid I_{i} \diamond I_{j} \}$. Then $\mathcal{D}$ is
a c.e.\ set and $(i,j)\in \mathcal{D}$ clearly implies $B_{i} \cap
B_{j} =\emptyset $.

Suppose $x,y\in X$ are such that $x\neq y$. Then there exists
$i,j\in \mathbb{N}$ such that $x\in B_{i} $, $y\in B_{j} $ and
$(i,j)\in \mathcal{D}$. Namely, choose a positive rational number
$r$ such that $2r<d(x,y)$ and choose $i,j\in \mathbb{N}$ such that
$(x,r)=(\lambda _{i} ,\rho _{i} )$ and $(y,r)=(\lambda _{j} ,\rho
_{j} )$. Then $i$ and $j$ are the desired numbers.

Let $$\mathcal{C}=\{(i,j)\in \mathbb{N}^{2}\mid I_{i} \subseteq
_{F}I_{j} \mbox{ and }(\lambda _{i} =\lambda _{j} \mbox{ or
}(\lambda _{i} \neq \lambda _{j} \mbox{ and }\lambda _{i} \in
\mathbb{Q}))\}.$$ In general, if $f,g:\mathbb{N}^{k} \rightarrow
\mathbb{Q}$ are computable functions, then the set $\{x\in
\mathbb{N}^{k} \mid f(x)=g(x)\}$ is computable. Therefore the sets
$\{(i,j)\in \mathbb{N}^{2}\mid \beta _{i} =\beta _{j} \}$ and
$\{(i,j)\in \mathbb{N}^{2}\mid \gamma _{i} =\gamma  _{j} \}$ are
computable and it follows that the set $\{(i,j)\in
\mathbb{N}^{2}\mid \alpha _{i} =\alpha  _{j} \}$ is computable.
The set $\{i\in \mathbb{N}\mid \alpha _{i} \in \mathbb{Q}\}$ is
also computable and since $\lambda _{i} =\alpha _{\tau _{1} (i)}$
for each $i\in \mathbb{N}$ we conclude that $\mathcal{C}$ is a
c.e.\ set.

If $(i,j)\in \mathcal{C}$, then obviously $B_{i} \subseteq B_{j}
$.

Suppose now that $i,j\in \mathbb{N}$ and $x\in B_{i} \cap B_{j} $.
We claim that there exists $k\in \mathbb{N}$ such that $x\in B_{k}
$, $(k,i)\in \mathcal{C}$ and $(k,j)\in \mathcal{C}$. We have two
cases.
\begin{enumerate}
\item[Case 1]: $x\neq \lambda _{i} $ or $x\neq \lambda _{j} $.
Then $x\in \mathbb{Q}$ and $x\in I_{i} \cap I_{j} $. By
Proposition \ref{svojstva-diamond-F}(4) there exists $k\in
\mathbb{N}$ such that $x\in I_{k} $, $I_{k} \subseteq _{F}I_{i} $,
$I_{k} \subseteq _{F}I_{j} $ and $\lambda _{k} \in \mathbb{Q}$. It
follows $x\in B_{k} $, $(k,i)\in \mathcal{C}$ and $(k,j)\in
\mathcal{C}$.

\item[Case 2]: $x=\lambda _{i} = \lambda _{j} $. Choose a positive
rational number $r$ such that $r<\rho _{i} $ and $r<\rho _{j} $.
Let $k\in \mathbb{N}$ be such that $(x,r)=(\lambda _{k} ,\rho _{k}
)$. Then $I_{k} \subseteq _{F}I_{i} $ and $I_{k} \subseteq
_{F}I_{j} $ and we conclude that $x\in B_{k} $, $(k,i)\in
\mathcal{C}$ and $(k,j)\in \mathcal{C}$.
\end{enumerate}

In particular, we have the following conclusion: if $i,j\in
\mathbb{N}$ and $x\in B_{i} \cap B_{j} $, then there exists $k\in
\mathbb{N}$ such that $x\in B_{k} $, $B_{k} \subseteq B_{i} $ and
$B_{k} \subseteq B_{j} $. This, together with the obvious fact
that $X=\bigcup_{i\in \mathbb{N}}B_{i} $, implies that there
exists a (unique) topology $\mathcal{T}$ on $X$ such that $\{B_{i}
\mid i\in \mathbb{N}\}$ is a basis for $\mathcal{T}$.

Then the triple $(X,\mathcal{T},(I_{i} ))$ is a computable
topological space: its characteristic relations are $\mathcal{C}$
and $\mathcal{D}$.

We claim that the topological space $(X,\mathcal{T})$ is not
regular. We have that $\mathbb{Q}$ is the union of all $B_{i} $
such that $\lambda _{i} \in \mathbb{Q}$. Therefore $\mathbb{Q}\in
\mathcal{T}$ and therefore $c+ \mathbb{Q}$ is a closed set in
$(X,\mathcal{T})$. Clearly $0\notin c+ \mathbb{Q}$.

Suppose $(X,\mathcal{T})$ is regular. Then there exist disjoint
sets $U,V\in \mathcal{T}$ such that $0\in U$ and $c+ \mathbb{Q}
\subseteq V$. It follows that there exists $i\in \mathbb{N}$ such
that $0\in B_{i} \subseteq U$. Hence $0\in I_{i} \cap
\mathbb{Q}\subseteq U$. So there exists an open interval $K$ in
$\mathbb{R}$ such that
\begin{equation}\label{KQ-U}
K\cap \mathbb{Q}\subseteq U.
\end{equation}
 Choose $x\in \mathbb{Q}$ such that
$c+x\in K$. Since $c+x\in V$, there exists $j\in \mathbb{N}$ such
that $c+x\in B_{j} \subseteq V$ and we conclude that there exists
an open interval $L$ in $\mathbb{R}$ such that $c+x\in L$ and
$L\cap \mathbb{Q}\subseteq V$. This and (\ref{KQ-U}) imply $(K
\cap L)\cap \mathbb{Q}=\emptyset $. But this is impossible since
$c+x\in K\cap L$: if two open intervals have a common point, then
they have a common rational point.

So $(X,\mathcal{T})$ is not regular.
\end{example1}

Suppose $(X,\mathcal{T},(I_{i} ))$  is a computable topological
space and $\mathcal{C}$ and $\mathcal{D}$ are its  characteristic
relations such that, beside the properties (1)--(4) from the
definition of characteristic relations, the following additional
properties hold:
\begin{enumerate}
\item[(5)] if $i,j\in \mathbb{N}$ are such that $(i,j)\in
\mathcal{D}$, then $(j,i)\in \mathcal{D}$;

\item[(6)] $(i,i)\in \mathcal{C}$ for each $i\in \mathbb{N}$ and
if $i,j,k\in \mathbb{N}$ are such that $(i,j)\in \mathcal{C}$ and
$(j,k)\in \mathcal{C}$, then $(i,k)\in \mathcal{C}$;

\item[(7)] if $i,j,k\in \mathbb{N}$ are such that $(k,i)\in
\mathcal{C}$ and $(i,j)\in \mathcal{D} $, then $(k,j)\in
\mathcal{D} $.
\end{enumerate}
Then we say that $\mathcal{C}$ and $\mathcal{D}$ are
\textbf{proper characteristic relations} for
$(X,\mathcal{T},(I_{i} ))$.

Every computable topological space has proper characteristic
relations. This is the contents of the following proposition.
\begin{proposition}
Let $(X,\mathcal{T},(I_{i} ))$ be a computable topological space.
Then there exist proper characteristic relations for
$(X,\mathcal{T},(I_{i} ))$.
\end{proposition}
\begin{proof}
We first show that there exist characteristic relations for
$(X,\mathcal{T},(I_{i} ))$ which satisfy properties (5) and (6)
above.

Let $\mathcal{C}$ and $\mathcal{D}$ be characteristic relations
for $(X,\mathcal{T},(I_{i} ))$. We define
$$\mathcal{D}'=\mathcal{D}\cup \{(i,j)\mid (j,i)\in
\mathcal{D}\}$$ and we define $\mathcal{C}'$ as the set of all
$(i,j)\in \mathbb{N}^{2}$ for which there exist $n\in \mathbb{N}$
and $a_{0} ,\dots ,a_{n}\in \mathbb{N}$  such that $i=a_{0}$ ,
$j=a_{n}$  and $(a_{l} ,a_{l+1})\in \mathcal{C}$ for each $l<n$.
Clearly, $\mathcal{D}'$ is c.e. On the other hand, the set
$$\Omega =\{a\in \mathbb{N}\mid ((a)_{l},(a)_{l+1})\in
\mathcal{C}\mbox{ for each }l<\overline{a}\}$$ is c.e.\ (recall the notation from Subsection \ref{subs-fr}) by
Proposition \ref{cfv}(5) since $\Omega =\{a\in \mathbb{N}\mid \Phi
(a)\subseteq \mathcal{C} \}$, where $\Phi :\mathbb{N}\rightarrow
\mathcal{P}(\mathbb{N}^{2})$ is the c.f.v.\ function defined by
$\Phi (a)=\{((a)_{l} ,(a)_{l+1})\mid l<\overline{a}\}$
(Proposition \ref{cfv}(4)). We have
$$\mathcal{C}'=\{(i,j)\in \mathbb{N}^{2}\mid\mbox{there exists
}a\in \mathbb{N}\mbox{ such that }i=(a)_{0},
j=(a)_{\overline{a}}\mbox{ and }a\in \Omega \}$$ and therefore
$\mathcal{C}'$ is c.e.

If $i,j\in \mathbb{N}$ are such that $(i,j)\in \mathcal{D}'$, then
clearly $I_{i} \cap I_{j} =\emptyset $ and if $(i,j)\in
\mathcal{C}'$, then $I_{i} \subseteq I_{j} $. Since
$\mathcal{D}\subseteq \mathcal{D}'$ and $\mathcal{C}\subseteq
\mathcal{C}'$, properties (3) and (4) from the definition of
characteristic relations are also satisfied for $\mathcal{C}'$ and
$\mathcal{D}'$. Hence these are characteristic relations for
$(X,\mathcal{T},(I_{i} ))$. It is immediate from their definitions
that $\mathcal{D}'$ is symmetric and $\mathcal{C}'$ is reflexive
and transitive, so properties (5) and (6) above are satisfied.

Suppose now that we have characteristic relations $\mathcal{C}$
and $\mathcal{D}$ for $(X,\mathcal{T},(I_{i} ))$ which satisfy (5)
and (6). We define $$\mathcal{D}'=\{(i,j)\in \mathbb{N}^{2}\mid
\exists k,l\in \mathbb{N}\mbox{ such that }(i,k)\in \mathcal{C},
(j,l)\in \mathcal{C}\mbox{ and }(k,l)\in \mathcal{D}\}.$$ It is
easy to check that $\mathcal{C}$ and $\mathcal{D}'$ are proper
characteristic relations for $(X,\mathcal{T},(I_{i} ))$.
\end{proof}

\section{Effective separation of compact sets} \label{sect-escs}
In this section
let $(X,\mathcal{T},(I_{i} ))$ be some fixed computable
topological space and let $\mathcal{C}$ and $\mathcal{D}$ be its
proper characteristic relations.


\begin{lemma} \label{(4)-opcenito}
Suppose $n\in \mathbb{N}$, $i_{0} ,\dots ,i_{n} \in \mathbb{N}$
and $x\in I_{i_{0}}\cap \dots \cap I_{i_{n} }$. Then there exists
$k\in \mathbb{N}$ such that $x\in I_{k} $ and $(k,i_{0} ),\dots
,(k,i_{n} )\in \mathcal{C}$.
\end{lemma}
\begin{proof}
Using reflexivity and transitivity of $\mathcal{C}$ and property
(4) from the definition of a computable topological space, this
follows easily by induction.
\end{proof}

Let $i,a\in \mathbb{N}$. We say that $I_{i} $ is
$\mathcal{C}$-\textbf{contained} in $J_{a} $ and write $I_{i}
\subseteq _{\mathcal{C}}J_{a} $ if there exists $j\in [a]$ such
that $(i,j)\in \mathcal{C}$. Obviously  $I_{i} \subseteq
_{\mathcal{C}}J_{a} $ implies $I_{i} \subseteq  J_{a}$.

Let $a,b\in \mathbb{N}$. We say that $J_{a}$ is
$\mathcal{C}$-\textbf{contained} in $J_{b}$ and write $J_{a}
\subseteq _{\mathcal{C}}J_{b}$ if $I_{i} \subseteq
_{\mathcal{C}}J_{b}$ for each $i\in [a]$. If $J_{a}\subseteq
_{\mathcal{C}}J_{b}$, then clearly $J_{a}\subseteq J_{b}$. Note
also that $J_{a}\subseteq  _{\mathcal{C}}J_{a}$ for each $a\in
\mathbb{N}$.

\begin{proposition} \label{K-Ja-Jb}
Suppose $K$ is a nonempty compact set in $(X,\mathcal{T})$ and
$a,b\in \mathbb{N}$ are such that $K\subseteq J_{a}\cap J_{b}$.
Then there exists $c\in \mathbb{N}$ such that $K\subseteq J_{c}$,
$J_{c}\subseteq _{\mathcal{C}}J_{a}$ and $J_{c}\subseteq
_{\mathcal{C}}J_{b}$.
\end{proposition}
\begin{proof}
Let $x\in K$. Then there exists $i\in [a]$ and $j\in [b]$ such
that $x\in I_{i} \cap I_{j} $. By definition of computable
topological space, there exists $k\in \mathbb{N}$ such that $x\in
I_{k} $ and  $(k,i),(k,j)\in \mathcal{C}$.

So for each $x\in K$ there exists $k_{x}\in \mathbb{N}$ such that
$x\in I_{k_{x}}$, $I_{k_{x}}\subseteq_ \mathcal{C}J_{a}$ and
$I_{k_{x}}\subseteq _\mathcal{C}J_{b}$. Since $\{I_{k_{x}}\mid
x\in K\}$ is an open cover of $K$, there exists $n\in \mathbb{N}$
and $x_{0} ,\dots ,x_{n} \in K$ such that $$K\subseteq I_{k_{x_{0}
}}\cup \dots \cup I_{k_{x_{n} }}.$$ Choose $c\in \mathbb{N}$ such
that $[c]=\{k_{x_{0} },\dots ,k_{x_{n} }\}$. Then $K\subseteq
J_{c}$, $J_{c}\subseteq _{\mathcal{C}}J_{a}$ and $J_{c}\subseteq
_{\mathcal{C}}J_{b}$.
\end{proof}

If $a,b,c\in \mathbb{N}$ are such that $J_{a}\subseteq
_{\mathcal{C}}J_{b}$ and $J_{b}\subseteq _{\mathcal{C}}J_{c}$,
then the transitivity of $\mathcal{C}$ easily gives
$J_{a}\subseteq _{\mathcal{C}}J_{c}$. Using this, we get the
following consequence of Proposition \ref{K-Ja-Jb}.

\begin{corollary} \label{K-Ja-Jb-cor}
Let $K$ be a nonempty compact set in $(X,\mathcal{T})$, $n\in
\mathbb{N}$ and $a_{0} ,\dots ,a_{n} \in \mathbb{N}$ such that
$K\subseteq J_{a_{0} }\cap \dots \cap J_{a_{n} }$. Then there
exists $c\in \mathbb{N}$ such that $K\subseteq J_{c}$ and
$J_{c}\subseteq _{\mathcal{C}}J_{a_{0} }$, \dots , $J_{c}\subseteq
_{\mathcal{C}}J_{a_{n} }$.
\end{corollary}

Let $i,a\in \mathbb{N}$. We say that $I_{i} $ and $J_{a} $ are
$\mathcal{D}$-\textbf{disjoint} and write $I_{i} \diamond
_{\mathcal{D}}J_{a} $ if $(i,j)\in \mathcal{D}$ for each $j\in
[a]$. Note that $I_{i} \diamond _{\mathcal{D}}J_{a} $ implies
$I_{i} \cap J_{a} =\emptyset $.

\begin{lemma} \label{x-notin-K-lem}
Suppose $K$ is a nonempty compact set in $(X,\mathcal{T})$ and
$x\in X\setminus K$. Then there exist $i,a\in \mathbb{N}$ such
that $x\in I_{i} $, $K\subseteq J_{a}$ and $I_{i}
\diamond_{\mathcal{D}}J_{a}$.
\end{lemma}
\begin{proof}
Let $y\in K$. Since $x\neq y$, by definition of computable
topological space there exist $i_{y},j_{y}\in \mathbb{N}$ such
that $x\in I_{i_{y}}$, $y\in I_{j_{y}}$ and $(i_{y},j_{y})\in
\mathcal{D}$. We have that $\{I_{j_{y}}\mid y\in K\}$ is an open
cover of $K$ and therefore there exist $n\in \mathbb{N}$ and
$y_{0} ,\dots ,y_{n} \in K$ such that
\begin{equation}\label{x-notin-K}
K\subseteq I_{j_{y_{0} }}\cup \dots \cup I_{j_{y_{n} }}.
\end{equation}
 On the other hand, $x\in
I_{i_{y_{0} }}\cap  \dots \cap  I_{i_{y_{n} }}$ and by Lemma
\ref{(4)-opcenito} there exists $k\in \mathbb{N}$ such that $x\in
I_{k} $ and $(k,i_{y_{0}}),\dots ,(k,i_{y_{n}})\in \mathcal{C}$.
Since $(i_{y_{0}},j_{y_{0}}),\dots ,(i_{y_{n}},j_{y_{n}})\in
\mathcal{C}$, by property (7) from the definition of proper
characteristic relations we have
\begin{equation}\label{x-notin-K-1}
(k,j_{y_{0}}),\dots ,(k,j_{y_{n}})\in \mathcal{D}.
\end{equation}

Choose $a\in \mathbb{N}$ so that $[a]=\{j_{y_{0}},\dots
,j_{y_{n}}\}$. Then $K\subseteq J_{a}$ by (\ref{x-notin-K}) and
$I_{k} \diamond _{\mathcal{D}}J_{a}$ by (\ref{x-notin-K-1}). Since
$x\in I_{k} $, this proves the lemma.
\end{proof}

Let $a,b\in \mathbb{N}$. We say that $J_{a}$ and $J_{b}$ are
$\mathcal{D}$-\textbf{disjoint} and write
$J_{a}\diamond_{\mathcal{D}}J_{b}$ if $(i,j)\in \mathcal{D}$ for
all $i\in [a]$ and $j\in [b]$. Clearly,
$J_{a}\diamond_{\mathcal{D}}J_{b}$ if and only if $I_{i}
\diamond_{\mathcal{D}}J_{b}$ for each $i\in [a]$. Note that
$J_{a}\diamond_{\mathcal{D}}J_{b}$ implies $J_{a}\cap
J_{b}=\emptyset $.

The following Lemma is a consequence of property (7) from the
definition of proper characteristic relations.
\begin{lemma} \label{(7)-J-J}
Let $i,a,b,c,d\in \mathbb{N}$.
\begin{enumerate}
\item[(i)] If $I_{i} \diamond_{\mathcal{D}} J_{a}$ and
$J_{b}\subseteq _{\mathcal{C}}J_{a}$, then $I_{i}
\diamond_{\mathcal{D}} J_{b}$;

\item[(ii)] If $J_{c} \diamond_{\mathcal{D}} J_{a}$ and
$J_{b}\subseteq _{\mathcal{C}}J_{a}$, then $J_{c}
\diamond_{\mathcal{D}} J_{b}$.

\item[(iii)] If $J_{c} \diamond_{\mathcal{D}} J_{a}$,
$J_{b}\subseteq _{\mathcal{C}}J_{a}$ and  $J_{d}\subseteq
_{\mathcal{C}}J_{c}$ then $J_{d} \diamond_{\mathcal{D}} J_{b}$.
\end{enumerate}
\end{lemma}

\begin{lemma}\label{K-L-disj}
Let $K$ and $L$ be nonempty disjoint compact sets in
$(X,\mathcal{T})$. Then there exists $a,b\in \mathbb{N}$ such that
$K\subseteq J_{a}$, $L\subseteq J_{b}$ and $J_{a}\diamond
_{\mathcal{D}}J_{b}$.
\end{lemma}
\begin{proof}
Let $x\in K$. By Lemma \ref{x-notin-K-lem} there exist
$i_{x},c_{x}\in \mathbb{N}$ such that $x\in I_{i_{x}}$,
$L\subseteq J_{c_{x}}$ and
$I_{i_{x}}\diamond_{\mathcal{D}}J_{c_{x}}$. Compactness of $K$
implies that there exist $x_{0} ,\dots ,x_{n} \in K$ such that
$$K\subseteq I_{i_{x_{0} }}\cup \dots \cup I_{i_{x_{n} }}.$$ We have
$L\subseteq J_{c_{x_{0} }}\cap \dots \cap J_{c_{x_{n} }}$ and by
Corollary \ref{K-Ja-Jb-cor} there exists $b\in \mathbb{N}$ such
that $L\subseteq J_{b}$ and $J_{b}\subseteq _{\mathcal{C}}
J_{c_{x_{0} }}$, \dots , $J_{b}\subseteq _{\mathcal{C}}
J_{c_{x_{n} }}$. We have $I_{i_{x_{0}
}}\diamond_{\mathcal{D}}J_{c_{x_{0} }}$, \dots , $I_{i_{x_{n}
}}\diamond_{\mathcal{D}}J_{c_{x_{n} }}$ and Lemma \ref{(7)-J-J}(i)
implies that $$I_{i_{x_{0} }}\diamond_{\mathcal{D}}J_{b},~ \dots
,~ I_{i_{x_{n} }}\diamond_{\mathcal{D}}J_{b}.$$ If we choose $a\in
\mathbb{N}$ such that $[a]=\{i_{x_{0} },\dots ,i_{x_{n} }\}$, then
we have $K\subseteq J_{a}$, $L\subseteq J_{b}$ and
$J_{a}\diamond_{\mathcal{D}}J_{b}$.
\end{proof}

\begin{theorem} \label{separacija-komp}
Let $\mathcal{F}$ be a finite family of nonempty compact sets in
$(X,\mathcal{T})$. Let $A$ be a finite subset of $\mathbb{N}$.
Then for each $K\in \mathcal{F}$ we can select $i_{K}\in
\mathbb{N}$ so that the following hold:
\begin{enumerate}
\item[(i)] $K\subseteq J_{i_{K}}$ for each $K\in \mathcal{F}$;

\item[(ii)] if $K,L\in \mathcal{F}$ are such that $K\cap
L=\emptyset $, then $J_{i_{K}}\diamond_{\mathcal{D}}J_{i_{L}}$;

\item[(iii)] if $K\in \mathcal{F}$ and $a\in A$ are such that
$K\subseteq J_{a}$, then $J_{i_{K}}\subseteq _{\mathcal{C}}J_{a}$.
\end{enumerate}
\end{theorem}
\begin{proof} Let us first notice that each compact set in
$(X,\mathcal{T})$ is contained in some $J_{j} $.

Let $K,L\in \mathcal{F}$. By Lemma  \ref{K-L-disj} there exist
$u_{(K,L)}, v_{(K,L)}\in \mathbb{N}$ such that $K\subseteq
J_{u_{(K,L)}}$, $L\subseteq J_{v_{(K,L)}}$ and such that
\begin{equation}\label{separacija-komp-1}
J_{u_{(K,L)}}\diamond_{\mathcal{D}}J_{v_{(K,L)}}\mbox{ if }K\cap
L=\emptyset .
\end{equation}

Let $K\in \mathcal{F}$. Observe the numbers $u_{(K,L)}$ and
$v_{(L,K)}$, where $L\in \mathcal{F}$, and the numbers $a\in A$
such that $K\subseteq J_{a}$. There are only finitely many such
numbers and so by Corollary \ref{K-Ja-Jb-cor} there exists
$i_{K}\in \mathbb{N}$ such that $K\subseteq J_{i_{K}}$ and
$J_{i_{K}}\subseteq _{\mathcal{C}}J_{u_{(K,L)}}$ for each $L\in
\mathcal{F}$, $J_{i_{K}}\subseteq _{\mathcal{C}}J_{v_{(L,K)}}$ for
each $L\in \mathcal{F}$ and $J_{i_{K}}\subseteq
_{\mathcal{C}}J_{a}$ for each $a\in A$ such that $K\subseteq
J_{a}$.

Then the numbers $i_{K}$, $K\in \mathcal{F}$, are the required
numbers. Properties (i) and (iii) clearly hold and if $K,L\in
\mathcal{F}$ are such that $K\cap L=\emptyset $, then from
$J_{i_{K}}\subseteq _{\mathcal{C}}J_{u_{(K,L)}}$,
$J_{i_{L}}\subseteq _{\mathcal{C}}J_{v_{(K,L)}}$ and
(\ref{separacija-komp-1}) we get
$J_{i_{K}}\diamond_{\mathcal{D}}J_{i_{L}}$ (Lemma
\ref{(7)-J-J}(iii)).
\end{proof}

\begin{proposition}\label{c.e.-CD}
Let $$\Omega _{1} =\{(i,a)\in \mathbb{N}^{2}\mid I_{i} \subseteq
_{\mathcal{C}}J_{a}\},~~\Omega _{2} =\{(a,b)\in \mathbb{N}^{2}\mid
J_{a} \subseteq _{\mathcal{C}}J_{b}\},$$ $$\Gamma _{1} =\{(i,a)\in
\mathbb{N}^{2}\mid I_{a} \diamond _{\mathcal{D}}J_{a}\},~~\Gamma
_{2} =\{(a,b)\in \mathbb{N}^{2}\mid J_{a} \diamond
_{\mathcal{D}}J_{b}\}.$$ Then $\Omega _{1} $, $\Omega _{2} $,
$\Gamma _{1}$ and $\Gamma _{2} $ are c.e.\ sets.
\end{proposition}
\begin{proof}
Let $i,a\in \mathbb{N}$. We have
\begin{equation*}\label{c.e.-CD-1}
(i,a)\in \Omega _{1} \Leftrightarrow \mbox{ there exists } j\in
\mathbb{N} \mbox{ such that }(i,j)\in \mathcal{C}\mbox{ and }j\in
[a].
\end{equation*}
 The set $\{(j,a)\in
\mathbb{N}^{2}\mid j\in [a]\}$ is computable and so $\Omega _{1} $ is c.e.

The function $\Phi :\mathbb{N}^{2}\rightarrow
\mathcal{P}(\mathbb{N}^{2})$ defined by $\Phi (a,b)=[a]\times
\{b\}$ is c.f.v.\ by Proposition \ref{cfv}(2). For all $a,b\in
\mathbb{N}$ we have $$(a,b)\in \Omega _{2} \Leftrightarrow I_{i}
\subseteq _{\mathcal{C}}J_{b}\mbox{ for each }i\in
[a]\Leftrightarrow \Phi (a,b)\subseteq \Omega _{1} $$ and it
follows from Proposition \ref{cfv}(5) that $\Omega _{2} $ is c.e.

In a similar way we get that $\Gamma _{1} $ and $\Gamma _{2} $ are
c.e.
\end{proof}

\section{Local computable enumerability} \label{sect-lce}

Let $(X,\mathcal{T},(I_i))$ be a computable topological space and
let $A$ and $S$ be subsets of $X$ such that $A \subseteq S$. We
say that $A$ is \textbf{computably enumerable up to} $S$ in
$(X,\mathcal{T},(I_i))$ if there exists a c.e.\ subset $\Omega$ of
$\mathbb{N}$ such that for each $i\in \mathbb{N}$ the following
implications hold:
$$ I_i \cap A \neq
\emptyset\implies  i\in \Omega $$ $$i\in \Omega \implies I_i \cap
S \neq \emptyset.$$

Note the following: if $S$ is a closed set in $(X,\mathcal{T})$,
then $S$ is c.e.\ in $(X,\mathcal{T},(I_{i} ))$ if and only if $S$
is c.e.\ up to $S$ in $(X,\mathcal{T},(I_{i} ))$.

\begin{proposition} \label{unija-ce-upto}
Let $(X,\mathcal{T},(I_i))$ be a computable topological space and
let $A_{0} ,\dots ,A_{n} $, $S_{0} ,\dots ,S_{n} $ be subsets of
$X$ such that $A_{i} $ is c.e.\ up to $S_{i} $ for each $i\in
\{0,\dots ,n\}$. Then $A_{0} \cup \dots \cup A_{n} $ is c.e.\ up
to $S_{0} \cup \dots \cup S_{n} $. In particular, if $A_{0} ,\dots
A_{n} $ are c.e.\ up to a set $S$, then $A_{0} \cup \dots \cup
A_{n} $ is c.e.\ up to $S$.
\end{proposition}
\begin{proof}
Let $\Omega _{0} ,\dots ,\Omega _{n} $ be c.e.\ subsets of
$\mathbb{N}$ such that for each $j\in \{0,\dots ,n\}$ and each
$i\in \mathbb{N}$ the following implications hold: $$(I_i \cap
A_{j}  \neq \emptyset\implies  i\in \Omega_{j}) \mbox{ and }(i\in
\Omega_{j}  \implies I_i \cap S_{j}  \neq \emptyset).$$ Then for
each $i\in \mathbb{N}$ we have
$$I_i \cap
(A_{0} \cup \dots \cup A_{n} )  \neq \emptyset\implies  i\in
\Omega _{0} \cup \dots \cup \Omega _{n} $$ and $$i\in \Omega _{0}
\cup \dots \cup \Omega _{n}  \implies I_i \cap (S_{0} \cup \dots
\cup S_{n} )\neq \emptyset.$$ The set $\Omega _{0} \cup \dots \cup
\Omega _{n} $ is c.e.\ and the claim follows.
\end{proof}

Let $(X,\mathcal{T}, (I_i))$ be a computable topological space and
 $S \subseteq X$.  Let $x \in S$.
 We say that  $S$ is \textbf{computably enumerable at}
 $x$ in $(X,\mathcal{T}, (I_i))$
 if there exists a neighborhood $N$ of $x$ in $S$
 such that $N$ is c.e.\ up to $S$. We say that $S$ is
 \textbf{locally computably enumerable} in $(X,\mathcal{T}, (I_i))$
 if $S$ is c.e.\ at $x$ for each $x\in S$.

 Each c.e.\ set in $(X,\mathcal{T},(I_{i} ))$ is clearly locally
 c.e.

\begin{proposition} \label{comp-loc-ce-comp}
Let $(X,\mathcal{T}, (I_i))$ be a computable topological space and
let $S$ be a locally c.e.\ set in $(X,\mathcal{T}, (I_i))$ such
that $S$ is compact in $(X,\mathcal{T})$. Then $S$ is c.e.\ in
$(X,\mathcal{T}, (I_i))$.
\end{proposition}
\begin{proof}
For each $x\in S$ let $N_{x}$ be a neighborhood of $x$ in $S$ such
that $N_{x}$ is c.e.\ up to $S$. The sets $N_{x}$, $x\in S$, are
not necessarily open in $S$, but their interiors (in $S$) form an
open cover of $S$ and since $S$ is compact, there exist $x_{0}
,\dots ,x_{n} \in S$ such that
\begin{equation}\label{lokce-komp-ce-1}
S=N_{x_{0} }\cup \dots \cup N_{x_{n} }.
\end{equation}
Each of the sets $N_{x_{0} }$, \dots , $N_{x_{n} }$ is c.e.\ up to
$S$ and it follows from Proposition \ref{unija-ce-upto} and
(\ref{lokce-komp-ce-1}) that $S$ is c.e.\ up to $S$. So $S$ is
c.e.\ (it is closed since it is compact and $(X,\mathcal{T})$ is
Hausdorff).
\end{proof}

\section{Semicomputable manifolds} \label{sect-scm}
In this section let $n\in
\mathbb{N}\setminus \{0\}$ be fixed.

For $i\in \{1,\dots ,n\}$ let 
\begin{align*}
A_i &= \left\{ (x_1, \ldots ,
x_n)\in [-2,2]^n \mid x_i = -2 \right\},\\
B_i &= \left\{ (x_1,
\ldots , x_n)\in [-2,2]^n \mid x_i = 2 \right\},\\
C_i &=
\left\{ (x_1, \ldots, x_n) \in [-4,4]^n \mid x_i \leq 1
\right\},\\
D_i &= \left\{ (x_1, \ldots, x_n) \in [-4,4]^n \mid
x_i \geq -1 \right\}.
\end{align*}

 We will use the
following nontrivial topological fact (see Theorem 5.1 in
\cite{lmcs:mnf}, Corollary 3.2 in \cite{lmcs:cell} and Theorem
1.8.1 in \cite{engel}).

\begin{theorem} \label{pokrivanje-kvadrata}
Suppose $U_{1} ,\dots ,U_{n} $ and $V_{1} ,\dots ,V_{n} $ are open
subsets of $\mathbb{R}^{n} $ such that
$$U_{i} \cap A_{i}=\emptyset ,~V_{i} \cap B_{i}=\emptyset
\mbox{ and }U_{i} \cap V_{i} =\emptyset $$ for each $i\in
\{1,\dots,n\}$. Then $$[-2,2]^{n}\nsubseteq  U_{1} \cup \dots \cup
U_{n} \cup V_{1} \cup \dots \cup V_{n}.$$
\end{theorem}

\begin{lemma} \label{S-minus-J}
Let $(X,\mathcal{T},(I_{i} ))$ be a computable topological space
and $S$ a semicomputable set in this space.
\begin{enumerate}
\item[(i)] Let $m\in \mathbb{N}$. The set $S\setminus J_{m} $ is
semicomputable in $(X,\mathcal{T},(I_{i} ))$.

\item[(ii)] Let $k\in \mathbb{N}\setminus \{0\}$. The set
$\{(j_{1} ,\dots ,j_{k} )\in \mathbb{N}^{k} \mid S\subseteq
J_{j_{1} }\cup \dots \cup J_{j_{k} }\}$ is c.e.
\end{enumerate}
\end{lemma}
\begin{proof}
Claim (i) can be proved in the same way as Lemma 3.3 in
\cite{lmcs:mnf}. For (ii), it is enough to prove that there exists
a computable function $\varphi :\mathbb{N}^{k} \rightarrow
\mathbb{N}$ such that $J_{j_{1} }\cup \dots \cup J_{j_{k}
}=J_{\varphi (j_{1} ,\dots ,j_{k} )}$ for all $j_{1} ,\dots
,j_{k}\in \mathbb{N}$. For this, it is enough to prove that there
exists a computable function $\varphi :\mathbb{N}^{2}\rightarrow
\mathbb{N}$ such that $J_{a}\cup J_{b}=J_{\varphi (a,b)}$ for all
$a,b\in \mathbb{N}$. The function $\mathbb{N}^{2}\rightarrow
\mathcal{P}(\mathbb{N})$, $(a,b)\mapsto [a]\cup [b]$ is c.f.v.\
(Proposition \ref{cfv}(1)) and for all
$a,b\in \mathbb{N}$ there exists $c\in \mathbb{N}$ such that
$[a]\cup [b]=[c]$. The set $\{(a,b,c)\in \mathbb{N}^{3}\mid
[a]\cup [b]=[c]\}$ is computable (Proposition \ref{cfv}(3)) and
therefore for all $a,b\in \mathbb{N}$ we can effectively find
$c\in \mathbb{N}$ such that $[a]\cup [b]=[c]$.
\end{proof}

In this paper we seek for conditions under which a semicomputable
set is computable. Equivalently, we seek for conditions under
which a semicomputable set is c.e. The next theorem is one of the
main results of the paper. It gives a sufficient condition that a
semicomputable set is c.e.\ at some point.

\begin{theorem} \label{lokalno-eukl}
Let $(X,\mathcal{T},(I_i))$ be a computable topological space, let
$S$ be a semicomputable set in this space and let $x \in S$.
Suppose that there exists a neighborhood of $x$ in $S$ which is
homeomorphic to some $\mathbb{R}^{n} $. Then $S$ is c.e.\ at $x$.
\end{theorem}
\begin{proof}
Let $N$ be a neighborhood of $x$ in $S$ which is homeomorphic to
$\mathbb{R}^{n} $. We may assume that $N$ is open in $S$ (as in
the proof of Theorem 5.6 in \cite{lmcs:mnf}). Let
$f:\mathbb{R}^{n} \rightarrow N$ be a homeomorphism. We may also
assume that $f(0)=x$.

For $a,b\in \mathbb{R}$ we will denote by $\langle a,b\rangle $
the open interval $\{x\in \mathbb{R}\mid a<x<b\}$. The set
$f\left( \langle -4,4 \rangle^n \right)$ is open in $N$ and
therefore it is open in $S$. It follows that $S\setminus f\left(
\langle -4,4 \rangle^n \right)$ is compact (it is closed in the
compact set $S$). This set is clearly disjoint with the compact
set $f([-2,2])^{n} $ and Lemma \ref{K-L-disj} implies that there
exists $m_0 \in \mathbb{N}$ such that
$$
S \setminus f\left( \langle -4,4 \rangle^n \right) \subseteq
J_{m_0} \mbox{ and } J_{m_0} \cap f([-2,2]^n) = \emptyset.
$$
Let $S' = S \setminus J_{m_0}$. By Lemma \ref{S-minus-J}(i) $S'$
is semicomputable in $(X,\mathcal{T}, (I_i))$ and we have
\begin{equation}\label{S'-4}
f([-2,2]^n) \subseteq S' \subseteq f\left( [ -4,4 ]^n \right).
\end{equation}

Let $i \in \{ 1, \ldots, n \}$. The sets  $A_i$, $B_i$, $C_i$ and
$D_i$ (defined at the begin of this section) are clearly compact
in $\mathbb{R}^{n} $ and we have $A_{i} \cap D_{i} =\emptyset $,
$B_{i} \cap C_{i} =\emptyset $. Therefore $f(A_i)$, $f(B_i)$,
$f(C_i)$ and $f(D_i)$ are compact in $(X,\mathcal{T})$ and
$f(A_{i}) \cap f(D_{i}) =\emptyset $, $f(B_{i}) \cap f(C_{i})
=\emptyset $. By Lemma \ref{K-L-disj} there exist $d_i$, $c_i \in
\mathbb{N}$ such that
\begin{equation}
\label{C-i-B} f(C_i) \subseteq J_{c_i} \text{ and } J_{c_i} \cap
f(B_i) = \emptyset,
\end{equation}
\begin{equation}
\label{D-i-A} f(D_i) \subseteq J_{d_i} \text{ and }J_{d_i} \cap
f(A_i) = \emptyset.
\end{equation}

Choose a computable function $\varphi \colon \mathbb{N} \to
\mathbb{N}$ such that $I_i = J_{\varphi(i)}$ for each $i \in
\mathbb{N}$ (such a function certainly exists).

Let us assume that $l \in \mathbb{N}$ is such that
$$
I_l \cap f \left( [-1,1]^n \right) \neq \emptyset.
$$
Then there exists $v \in [-1,1]^n$, $v = (v_1, \ldots, v_n)$, such
that  $f(v) \in I_l$ and so $v \in f^{-1}(I_l)$. Since
$f^{-1}(I_l)$ is open in $\mathbb{R}^{n} $, there exists
$\epsilon > 0$ such that
\begin{equation}\label{E-v1-vn}
[v_1 - \epsilon, v_1 + \epsilon] \times \ldots \times [v_n -
\epsilon, v_n + \epsilon] \subseteq f^{-1}(I_l).
\end{equation}
We may assume $\epsilon < 1$. Let $ E = [v_1 - \epsilon, v_1 +
\epsilon] \times \ldots \times [v_n - \epsilon, v_n + \epsilon] $.
By (\ref{E-v1-vn}) we have
 $f(E) \subseteq I_l$, i.e.
\begin{equation}\label{E-fi-l}
f(E) \subseteq J_{\varphi(l)}.
\end{equation}

For $i \in \{1, \ldots, n \}$ let
$$\tilde{A}_i = \left\{ (x_1, \ldots, x_n) \in [-4,4]^n \mid x_i \leq v_i - \epsilon
\right\},$$ $$\tilde{B}_i = \left\{ (x_1, \ldots, x_n) \in
[-4,4]^n \mid x_i \geq v_i + \epsilon \right\}.$$

Note that
\begin{equation}
\label{eq: eq3 iz thm: S poluizracunljiv i postoji otv okolina N
od x iz S td je N cong Rn tada je S izracunljivo prebrojiv u tocki
x} \tilde{A}_i \subseteq C_i \text{ and } \tilde{B}_i \subseteq
D_i
\end{equation}
for each $i \in \{1, \ldots, n\}$. Furthermore
$$
\tilde{A}_1 \cup \tilde{B}_1 \cup \ldots \cup \tilde{A}_n \cup
\tilde{B}_n \cup E = [-4,4]^n
$$
and so
\begin{equation}
\label{eq: eq4 iz thm: S poluizracunljiv i postoji otv okolina N
od x iz S td je N cong Rn tada je S izracunljivo prebrojiv u tocki
x} f(\tilde{A}_1) \cup f(\tilde{B}_1) \cup \ldots \cup
f(\tilde{A}_n) \cup f(\tilde{B}_n) \cup f(E) = f \left( [-4,4]^n
\right).
\end{equation}
For each $i \in \{ 1, \ldots, n\}$ we have $\tilde{A}_i \cap
\tilde{B}_i = \emptyset$, thus
\begin{equation}\label{Atilda-disj}
f(\tilde{A}_i) \cap f(\tilde{B}_i) = \emptyset.
\end{equation}

By (\ref{eq: eq3 iz thm: S poluizracunljiv i postoji otv okolina N
od x iz S td je N cong Rn tada je S izracunljivo prebrojiv u tocki
x}) for each $i \in \{1, \ldots, n \}$ we have $f(\tilde{A}_i)
\subseteq f(C_i)$ and $f(\tilde{B}_i) \subseteq f(D_i)$ which,
together with (\ref{C-i-B}) and (\ref{D-i-A}), gives
\begin{equation}\label{Atilda-Ci}
f(\tilde{A}_i) \subseteq J_{c_i} \text{ and } f(\tilde{B}_i)
\subseteq J_{d_i}.
\end{equation}
The sets $f(\tilde{A}_1), \ldots, f(\tilde{A}_n),f(\tilde{B}_1),
\ldots, f(\tilde{B}_n), f(E)$ are nonempty and compact in $(X,
\mathcal{T})$. By Theorem \ref{separacija-komp},
(\ref{Atilda-disj}), (\ref{Atilda-Ci}) and (\ref{E-fi-l}) there
exist numbers $a_1, \ldots, a_n,b_1, \ldots, b_n, e  \in
\mathbb{N}$ such that for each $i \in \{ 1, \ldots, n \}$
$$f(\tilde{A}_i) \subseteq J_{a_i},~~ f(\tilde{B}_i) \subseteq J_{b_i},~~ f(E) \subseteq
J_e,$$
$$J_{a_i} \subseteq _{\mathcal{C}} J_{c_i},~~ J_{b_i} \subseteq
_{\mathcal{C}} J_{d_i},~~ J_e \subseteq _{\mathcal{C}}
J_{\varphi(l)} \text{ and }
J_{a_i}\diamond_{\mathcal{D}}J_{b_i}.$$

It follows from (\ref{S'-4}) and (\ref{eq: eq4 iz thm: S
poluizracunljiv i postoji otv okolina N od x iz S td je N cong Rn
tada je S izracunljivo prebrojiv u tocki x}) that
$$
S' \subseteq J_{a_1} \cup J_{b_1} \cup \ldots \cup J_{a_n} \cup
J_{b_n} \cup J_e.
$$

We have proved the following: if $l \in \mathbb{N}$ is such that
$I_l \cap f([-1,1]^n) \neq \emptyset$, then there exist $a_1,
\ldots, a_n, b_1, \ldots, b_n, e \in \mathbb{N}$ such that
\begin{enumerate}
\item[(1)] $J_{a_i} \subseteq _{\mathcal{C}} J_{c_i}$ for each $i
\in \{ 1, \ldots, n \}$;

\item[(2)] $J_{b_i}\subseteq _{\mathcal{C}}  J_{d_i}$ for each $i
\in \{1, \ldots, n \}$;

\item[(3)] $J_e \subseteq _{\mathcal{C}}  J_{\varphi(l)}$
\item[(4)] $J_{a_i}\diamond_{\mathcal{D}}J_{b_i}$ for each $ i \in
\{ 1, \ldots, n \}$

\item[(5)] $S' \subseteq J_{a_1} \cup J_{b_1} \cup \ldots \cup
J_{a_n} \cup J_{b_n} \cup J_{e}$.
\end{enumerate}

Let $\Gamma$ be the set of all $(l,a_1, \ldots, a_n,b_1, \ldots,
b_n,e) \in \mathbb{N}^{2n+2}$ such that (1) - (5) hold.
Furthermore, let $\Omega$ be the set of all $l\in \mathbb{N}$ for
which there exist  $a_1, \ldots, a_n, b_1, \ldots, b_n, e \in
\mathbb{N}$ such that $$(l,a_1,\ldots,a_n,b_1,\ldots,b_n,e) \in
\Gamma.$$ Note that for each $l\in \mathbb{N}$ we have the
following implication
$$I_l \cap f([-1,1]^n) \neq \emptyset \implies l \in \Omega.$$

Using Proposition \ref{c.e.-CD} and Lemma \ref{S-minus-J}(2) we
easily conclude that $\Gamma $ is a c.e.\ set as the intersection
of finitely many c.e.\ sets. It follows that $\Omega $ is also
c.e.

We now prove the following: if $l\in \Omega $, then $I_l \cap S
\neq \emptyset$.

Suppose $l \in \Omega$. Then there exist $a_1, \ldots, a_n,b_1,
\ldots, b_n, e \in \mathbb{N}$ such that \\
 $(l,a_1, \ldots, a_n, b_1, \ldots, b_n, e) \in \Gamma$.
 So, for the numbers $l,a_1, \ldots, a_n, b_1, \ldots, b_n, e$
 statements  (1)--(5) hold.

Since $f([-2,2]^n) \subseteq S'$, by (5) we have $$ f([-2,2]^n)
\subseteq J_{a_1} \cup J_{b_1} \cup \ldots \cup J_{a_n} \cup
J_{b_n} \cup J_{e}
$$
and it follows
\begin{equation}
\label{eq:eq5-iz-thm} [-2,2]^n \subseteq f^{-1}(J_{a_1})
\cup f^{-1}(J_{b_1}) \cup \ldots \cup
f^{-1}(J_{a_n}) \cup f^{-1}(J_{b_n})\cup
f^{-1}(J_e).
\end{equation}

Let $i \in \{1, \ldots, n \}$. It follows from (1) and
(\ref{C-i-B}) that $J_{a_i} \cap f(B_i) = \emptyset$ and therefore
\begin{equation}
\label{eq:eq6-iz-thm} f^{-1}(J_{a_i}) \cap B_i =
\emptyset.
\end{equation}
Similarly, from (2) and (\ref{D-i-A}) we get
\begin{equation}
\label{eq:eq7-iz-thm} f^{-1}(J_{b_i}) \cap A_i =
\emptyset.
\end{equation}
By (4) we have
\begin{equation}
\label{eq:eq8-iz-thm} f^{-1}(J_{a_i}) \cap
f^{-1}(J_{b_i}) = \emptyset.
\end{equation}
The sets $f^{-1}(J_{a_1}), \ldots,
f^{-1}(J_{a_n}), f^{-1}(J_{b_1}), \ldots,
f^{-1}(J_{b_n})$ are open in $\mathbb{R}$. It follows from
(\ref{eq:eq6-iz-thm}), (\ref{eq:eq7-iz-thm}),
(\ref{eq:eq8-iz-thm}) and Theorem \ref{pokrivanje-kvadrata} that
$$
[-2,2]^n \nsubseteq f^{-1}(J_{a_1}) \cup
f^{-1}(J_{b_1}) \cup \ldots \cup f^{-1}(J_{a_n})
\cup f^{-1}(J_{b_n}).
$$
This and (\ref{eq:eq5-iz-thm}) give
$$
[-2,2]^n \cap f^{-1}(J_e) \neq \emptyset.
$$
So $f([-2,2]^n) \cap J_e \neq \emptyset$ and (3) implies
$f([-2,2]^n) \cap J_{\varphi(l)} \neq \emptyset$. Hence $I_l \cap
f([-2,2]^n) \neq \emptyset$ and therefore $I_l \cap S \neq
\emptyset$.

We have proved that for each $l \in \mathbb{N}$ the following
implications hold:
\begin{enumerate}
\item[$i$)] $I_l \cap f([-1,1]^n) \neq \emptyset \Rightarrow l \in
\Omega$ \item[$ii$)] $l \in \Omega \Rightarrow I_l \cap S \neq
\emptyset$.
\end{enumerate}
Since $f([-1,1]^n)$ is a neighborhood of $x$ u $S$, this proves
the theorem.
\end{proof}

Let $n \in \mathbb{N} \setminus \{ 0 \}$. A
topological space $X$ is said to be an $n$--\textbf{manifold} if
each point $x \in X$ has a neighborhood in $X$ which is
homeomorphic to $\mathbb{R}^{n} $.

\begin{theorem} \label{manifold-tm}
Let $(X,\mathcal{T}, (I_i))$ be a computable topological space and
let $S$ be a semicomputable set in this space which is, as a
subspace $(X,\mathcal{T})$, a manifold. Then $S$ is a computable
set in $(X,\mathcal{T}, (I_i))$.
\end{theorem}
\begin{proof}
By Theorem \ref{lokalno-eukl} $S$ is locally c.e. By Proposition
\ref{comp-loc-ce-comp} $S$ is c.e. So $S$ is computable.
\end{proof}

\section{Semicomputable manifolds with computable boundaries} \label{sect-scm-wb}
For $n \in \mathbb{N} \setminus \{ 0 \}$ let
$$
\mathbb{H}^n = \left\{ (x_1, \ldots, x_n) \in \mathbb{R}^n \mid
x_n \geq 0 \right\}
$$
and
$$
\Bd \mathbb{H}^n = \left\{ (x_1, \ldots, x_n) \in \mathbb{R}^n
\mid x_n = 0 \right\}.
$$

A
topological space $X$ is said to be an
$n$--\textbf{manifold with boundary} if for each $x \in X$ there
exists a neighborhood $N$ of $x$ in $X$ such that one of the
following holds:
\begin{enumerate}
\item[(1)] $N$ is homeomorphic to $\mathbb{R}^n$;
\item[(2)] there
exists a homeomorphism $f \colon \mathbb{H}^n \to N$ such that $x
\in f(\Bd \mathbb{H}^n)$.
\end{enumerate}

If $X$ is an $n$--manifold with boundary, we define $\partial X$
to be the set of all $x \in X$ such that $x$ has a neighborhood
$N$ with property (2). We say that $\partial X$ is the
\textbf{boundary} of the manifold $X$.

Each manifold is clearly a manifold with boundary. Conversely, if
$X$ is a manifold with boundary and $\partial X=\emptyset $, then
$X$ is a manifold.

It can be shown (see \cite{mu}) that if a point $x$ in a
topological space $X$ has a neighborhood which satisfies (1), then
it cannot have a neighborhood which satisfies (2). So a manifold
with boundary $X$ is a manifold if and only if $\partial
X=\emptyset $.

In order to prove that a semicomputable manifold $S$ with
computable boundary is computable, we will need Theorem
\ref{lokalno-eukl}, but we will also need an analogue of this
theorem which deals with points from $\partial S$ (Theorem
\ref{lokalno-eukl-rub}). First, we have a lemma.

\begin{lemma}
\label{prop: A_i, A_n. B_i prije teoream sa rubom} Let $n \in
\mathbb{N} \setminus \{ 0 \}$. For $i \in \{ 1, \ldots, n \}$ let
$$
B_i = \{ (x_1, \ldots, x_n) \in [-2,2]^{n-1} \times [0,2] \mid x_i
= 2 \}.
$$
For $i \in \{ 1, \ldots, n-1 \}$ let
$$
A_i = \{ (x_1, \ldots, x_n) \in [-2,2]^{n-1} \times [0,2] \mid x_i
= -2 \}
$$
and let
$$
A_n = \{ (x_1, \ldots, x_n) \in [-2,2]^{n-1} \times [0,2] \mid x_n
= 0 \}.
$$
Then there exist no open subsets $U_1, \ldots, U_n, V_1, \ldots,
V_n$ of $\mathbb{H}^n$ such that
\begin{equation}
\label{eq: eq1 propozicija A_i, A_n. B_i prije teoream sa rubom}
U_i \cap B_i = \emptyset,~~ V_i \cap A_i = \emptyset ~~\mbox{ and
} ~~ U_i \cap V_i = \emptyset
\end{equation}
for each $i \in \{ 1, \ldots, n \}$ and such that
\begin{equation}
\label{eq: eq2 propozicija A_i, A_n. B_i prije teoream sa rubom}
[-2,2]^{n-1} \times [0,2] \subseteq U_1 \cup \ldots \cup U_n \cup
V_1 \cup \ldots \cup V_n.
\end{equation}
\end{lemma}

\begin{proof}
Suppose the opposite, i.e.\ suppose that there exist sets
 $U_1, \ldots, U_n, V_1, \ldots,
V_n$ with the above properties.

Let $f \colon \mathbb{R} \to [0,\infty\rangle $ be defined by
$$
f(x) =
\begin{cases}
 \frac{x+2}{2}, & x \geq -2 \\
 0, & x \leq -2.
\end{cases}
$$
Let $\gamma \colon \mathbb{R}^n \to \mathbb{H}^n$ be defined by
$$
\gamma(z_1, \ldots,z_{n-1}, z_n) =(z_{1} ,\ldots, z_{n-1},f(z_{n}
)).$$ Since $f$ is continuous, $\gamma $ is also continuous. We
have
$$
\gamma([-2,2]^n) = [-2,2]^{n-1} \times [0,2].
$$
From this and (\ref{eq: eq2 propozicija A_i, A_n. B_i prije
teoream sa rubom}) it follows
\begin{equation} \label{eq: eq4
propozicija A_i, A_n. B_i prije teoream sa rubom} [-2,2]^n
\subseteq \gamma^{-1}(U_1) \cup \ldots \cup
\gamma^{-1} (U_n) \cup \gamma^{-1}(V_1) \cup
\ldots \cup \gamma^{-1}(V_n).
\end{equation}

By (\ref{eq: eq1 propozicija A_i, A_n. B_i prije teoream sa
rubom}) for each $i \in \{ 1, \ldots, n \}$ we have
\begin{equation}
\label{eq: eq5 propozicija A_i, A_n. B_i prije teoream sa rubom}
\gamma^{-1}(U_i) \cap \gamma^{-1}(V_i) =
\emptyset.
\end{equation}

For $i \in \{ 1, \ldots, n\}$ let
\begin{align*}
\tilde{A}_i &=
\left\{ (x_1, \ldots, x_n) \in [-2,2]^n \mid x_i = -2 \right\}\\
\tilde{B}_i &= \left\{ (x_1, \ldots, x_n) \in [-2,2]^n \mid x_i =
2 \right\}.
\end{align*}

Let $i \in \{ 1, \ldots, n \}$. We have
$$
\gamma(\tilde{B}_i) \subseteq B_i ~~\mbox{ and } ~~
\gamma(\tilde{A}_i) \subseteq A_i.
$$
Since $U_i \cap B_i = \emptyset$, we have $U_i \cap
\gamma(\tilde{B}_i) = \emptyset$ and consequently
\begin{equation}
\label{eq: eq6 propozicija A_i, A_n. B_i prije teoream sa rubom}
\gamma^{-1}(U_i) \cap \tilde{B}_i = \emptyset.
\end{equation}
Also $V_i \cap A_i = \emptyset$ implies
\begin{equation}
\label{eq: eq7 propozicija A_i, A_n. B_i prije teoream sa rubom}
\gamma^{-1}(V_i) \cap \tilde{A}_i = \emptyset.
\end{equation}

Since $\gamma$ is continuous, the sets $\gamma^{-1}(U_1),
\ldots, \gamma^{-1} (U_n), \gamma^{-1}(V_1),
\ldots, \gamma^{-1}(V_n)$ are open in $\mathbb{R}^n$.
This, together with (\ref{eq: eq4 propozicija A_i, A_n. B_i prije
teoream sa rubom}), (\ref{eq: eq5 propozicija A_i, A_n. B_i prije
teoream sa rubom}), (\ref{eq: eq6 propozicija A_i, A_n. B_i prije
teoream sa rubom}) and (\ref{eq: eq7 propozicija A_i, A_n. B_i
prije teoream sa rubom}) contradicts Theorem
\ref{pokrivanje-kvadrata}.
\end{proof}

\begin{theorem}
\label{lokalno-eukl-rub} Let $(X, \mathcal{T}, (I_i))$ be a
computable topological space. Let $S$ and $T$ be semicomputable
sets in this space such that  $T \subseteq S$ and let $x \in S$.
Let us suppose that there exists a neighborhood $N$ of $x$ in $S$
and a homeomorphism $f \colon \mathbb{H}^n \to N$ (for some $n \in
\mathbb{N}$) such that $$x \in f(\Bd \mathbb{H}^n)~\mbox{ and
}~f(\Bd \mathbb{H}^n) = N \cap T.$$ Then $S$ is c.e.\ at $x$.
\end{theorem}

\begin{proof}
It is known that each open ball in $\mathbb{H}^n$ (with respect to
the Euclidean metric on $\mathbb{H}^n$) centered at a point in
$\Bd \mathbb{H}^n$ is homeomorphic to $\mathbb{H}^n$. Therefore,
we may assume that $N$ is an open neighborhood of $x$ in $S$.

We may also assume that $x = f(0, \ldots, 0)$.

As in the proof of Theorem \ref{lokalno-eukl} we conclude that the
set $S \setminus f(\langle -4,4 \rangle^{n-1} \times [ 0,4 \rangle
)$ is compact in $(X, \mathcal{T})$. The set $f([-2,2 ]^{n-1}
\times [ 0,2 ])$ is also compact in $(X, \mathcal{T})$. These two
sets are disjoint and Lemma \ref{K-L-disj} implies that there
exists  $m_0 \in \mathbb{N}$ such that
\begin{equation}
\label{eq: eq1 teorem sa RUBOM jednakosti sa Jm0} S \setminus
f(\langle -4,4 \rangle^{n-1} \times [ 0,4 \rangle) \subseteq
J_{m_0} \text{~ and ~} J_{m_0} \cap f([-2,2 ]^{n-1} \times [ 0,2
]) = \emptyset.
\end{equation}
Let
$$
S' = S \setminus J_{m_0},
$$
$$
T' = T \setminus J_{m_0}.
$$
By Lemma \ref{S-minus-J}(i) the sets $S'$ and $T'$ are
semicomputable. We have
\begin{equation}
\label{eq: eq11 teorem sa RUBOM} f([-2,2 ]^{n-1} \times [ 0,2 ])
\subseteq S' \subseteq f([-4,4 ]^{n-1} \times [ 0,4 ]).
\end{equation}

For $i \in \{ 1, \ldots, n-1 \}$ let
\begin{align*}
C_i &=
\left\{ (x_1, \ldots, x_n) \in [-4,4 ]^{n-1} \times [ 0,4 ] \mid x_i \leq 1 \right\},\\
D_i &=
\left\{ (x_1, \ldots, x_n) \in [-4,4 ]^{n-1} \times [ 0,4 ] \mid x_i \geq -1 \right\},\\
A_i &=
\left\{ (x_1, \ldots, x_n) \in [-2,2 ]^{n-1} \times [ 0,2 ] \mid x_i = -2 \right\},\\
B_i &= \left\{ (x_1, \ldots, x_n) \in [-2,2 ]^{n-1} \times [ 0,2 ]
\mid x_i = 2 \right\}
\end{align*}
and let
\begin{align*}
C_n &=
\left\{ (x_1, \ldots, x_n) \in [-4,4 ]^{n-1} \times [ 0,4 ] \mid x_n \leq 1 \right\},\\
B_n &= \left\{ (x_1, \ldots, x_n) \in [-2,2 ]^{n-1} \times [ 0,2 ]
\mid x_n = 2 \right\}.
\end{align*}
These sets are clearly compact in $\mathbb{H}^{n} $. Therefore,
the sets $f(A_i)$, $f(B_i)$, $f(C_i)$, $f(D_i)$, for $i \in \{ 1,
\ldots, n-1 \}$, and $f(C_n)$ and $f(B_n)$ are compact in $(X,
\mathcal{T})$. Moreover, for each $i \in \{ 1, \ldots, n-1 \}$ we
have $f(A_i) \cap f(D_i) = \emptyset$ and $f(B_i) \cap f(C_i) =
\emptyset$. Also $f(B_n) \cap f(C_n) = \emptyset$.

By Lemma  \ref{K-L-disj} for each $i \in \{ 1, \ldots, n-1 \}$
there exist $d_i$, $c_i \in \mathbb{N}$ such that
\begin{equation}
\label{eq: eq2 teorem sa RUBOM} f(C_i) \subseteq J_{c_i} \text{~
and ~} J_{c_i} \cap f(B_i) = \emptyset,
\end{equation}
\begin{equation}
\label{eq: eq3 teorem sa RUBOM} f(D_i) \subseteq J_{d_i} \text{~
and ~}J_{d_i} \cap f(A_i) = \emptyset
\end{equation}
and there exists $c_n \in \mathbb{N}$ such that
\begin{equation}
\label{eq: eq4 teorem sa RUBOM} f(C_n) \subseteq J_{c_n} \text{~
and ~} J_{c_n} \cap f(B_n) = \emptyset.
\end{equation}
Let $\varphi \colon \mathbb{N} \to \mathbb{N}$ be a computable
function such that $I_i = J_{\varphi(i)}$ for each $i \in
\mathbb{N}$.

Suppose $l \in \mathbb{N}$ is such that
\begin{equation}\label{Iisijece}
I_l \cap f \left( [-1,1]^{n-1} \times [0,1] \right) \neq
\emptyset.
\end{equation}
Then $$f^{-1}( I_l) \cap \left(  [-1,1]^{n-1} \times
[0,1]\right) \neq \emptyset$$ and since $f^{-1}(I_l)$ is
open in $\mathbb{H}^n$ we may choose $v \in [-1,1]^{n-1} \times
[0,1]$, $v=(v_1, \ldots, v_n)$, such that $v_{n} >0$ and $v \in
f^{-1}(I_l)$. The fact that $f^{-1}(I_l)$ is open
implies that there exists $\epsilon>0$ such that $\epsilon<v_{n} $
and
$$
[v_1 - \epsilon, v_1 + \epsilon] \times \ldots \times [v_n -
\epsilon, v_n + \epsilon] \subseteq f^{-1}(I_l).
$$
Let
$$
E = [v_1 - \epsilon, v_1 + \epsilon] \times \ldots \times [v_n -
\epsilon, v_n + \epsilon].
$$
So $E \subseteq f^{-1}(I_l)$ and $f(E) \subseteq I_l$.

For $i \in \{1, \ldots, n \}$ let
\begin{align*}
\tilde{A}_i = \left\{ (x_1, \ldots, x_n) \in [-4,4]^{n-1} \times [0,4] \mid x_i \leq v_i - \epsilon \right\},\\
\tilde{B}_i = \left\{ (x_1, \ldots, x_n) \in [-4,4]^{n-1} \times
[0,4] \mid x_i \geq v_i + \epsilon \right\}.
\end{align*}

\noindent For each $i \in \{ 1, \ldots, n-1 \}$ we have
\begin{equation}
\label{eq: eq8 teorem sa RUBOM} \tilde{A}_i \subseteq C_i,~~
\tilde{B}_i \subseteq D_i \text{~and~} \tilde{A}_n \subseteq C_n.
\end{equation}
Furthermore, we have
$$
\tilde{A}_1 \cup \tilde{B}_1 \cup \ldots \cup \tilde{A}_n \cup
\tilde{B}_n \cup E = [-4,4]^{n-1} \times [0,4]
$$
and so
\begin{equation}
\label{eq: eq10 teorem sa RUBOM} f(\tilde{A}_1) \cup
f(\tilde{B}_1) \cup \ldots \cup f(\tilde{A}_n) \cup f(\tilde{B}_n)
\cup f(E) = f \left( [-4,4]^{n-1} \times [0,4] \right).
\end{equation}
For each $i \in \{ 1, \ldots, n\}$ we have $f(\tilde{A}_i) \cap
f(\tilde{B}_i) = \emptyset$ since $\tilde{A}_i \cap \tilde{B}_i =
\emptyset$.

Let $i \in \{ 1, \ldots, n-1 \}$. It follows from (\ref{eq: eq8
teorem sa RUBOM}) that $f(\tilde{A}_i) \subseteq f(C_i)$ and
$f(\tilde{B}_i) \subseteq f(D_i)$ and so (\ref{eq: eq2 teorem sa
RUBOM}) and (\ref{eq: eq3 teorem sa RUBOM}) imply
$$
f(\tilde{A}_i) \subseteq J_{c_i} \text{~and~} f(\tilde{B}_i)
\subseteq J_{d_i}.
$$
By (\ref{eq: eq8 teorem sa RUBOM}) we have $f(\tilde{A}_n)
\subseteq f(C_n)$ and from (\ref{eq: eq4 teorem sa RUBOM}) we get
$f(\tilde{A}_n) \subseteq J_{c_n}$.

We have
\begin{align*}
f(\tilde{B}_n) \cap T' &=
\left( f(\tilde{B}_n) \cap N \right) \cap T'\\
&\subseteq
\left( f(\tilde{B}_n) \cap N \right) \cap T\\
&=
f(\tilde{B}_n) \cap \left( N \cap T \right)\\
&=
f(\tilde{B}_n) \cap f(\Bd \mathbb{H}^n)\\
&=
f(\tilde{B}_n \cap \Bd \mathbb{H}^n)\\
&=
f(\emptyset)\\
&= \emptyset.
\end{align*}
Hence
\begin{equation}
\label{eq: eq9 teorem sa RUBOM} f(\tilde{B}_n) \cap T' =
\emptyset.
\end{equation}

\noindent Let
$$
A_n = \left\{ (x_1, \ldots, x_n) \in [-2,2]^{n-1} \times [0,2]
\mid x_n = 0 \right\}.
$$
Since $A_n \subseteq \Bd \mathbb{H}^n$, we have
$$
f(A_n) \subseteq f(\Bd \mathbb{H}^n) = N \cap T.
$$
So $f(A_n) \subseteq T$. Furthermore   $f(A_n) \cap J_{m_0} =
\emptyset$ by (\ref{eq: eq1 teorem sa RUBOM jednakosti sa Jm0}).
Therefore $f(A_n) \subseteq T \setminus J_{m_0}$, i.e.
\begin{equation}
\label{eq: eq15 teorem sa RUBOM} f(A_n) \subseteq T'.
\end{equation}
In particular, $T'$ is a nonempty set. It follows from (\ref{eq:
eq9 teorem sa RUBOM}) and Lemma \ref{K-L-disj} that there exist
$d_n, t \in \mathbb{N}$ such that $f(\tilde{B}_n) \subseteq
J_{d_n}$, $T' \subseteq J_t$ and such that
$J_{d_n}\diamond_{\mathcal{D}}J_t$.

The sets $\tilde{A}_1, \ldots, \tilde{A}_n,\tilde{B}_1, \ldots,
\tilde{B}_n, E$ are nonempty and compact in $\mathbb{H}^n$.
Consequently, the sets $f(\tilde{A}_1), \ldots,
f(\tilde{A}_n),f(\tilde{B}_1), \ldots, f(\tilde{B}_n), f(E)$ are
nonempty and compact in $(X, \mathcal{T})$. Since $f(E) \subseteq
I_l$, we have $f(E) \subseteq J_{\varphi(l)}$.

By Theorem \ref{separacija-komp} there exist $a_1, \ldots,
a_n,b_1, \ldots, b_{n}, e \in \mathbb{N}$ such that for each $i
\in \{ 1, \ldots, n \}$ the following holds:
\begin{align*}
& f(\tilde{A}_i) \subseteq J_{a_i},~~ f(\tilde{B}_i) \subseteq J_{b_i},~~ f(E) \subseteq J_e, \\
& J_{a_i} \subseteq_{\mathcal{C}} J_{c_i},~~ J_{b_i}
\subseteq_{\mathcal{C}} J_{d_i},~~ J_e \subseteq_{\mathcal{C}}
J_{\varphi(l)} \text{~ and ~} J_{a_i}\diamond_{\mathcal{D}}
J_{b_i}.
\end{align*}
Since $J_{b_n} \subseteq_{\mathcal{C}} J_{d_n}$ and
$J_{d_n}\diamond_{\mathcal{D}}J_t$, we have
$J_{b_n}\diamond_{\mathcal{D}}J_t$.

It follows from (\ref{eq: eq11 teorem sa RUBOM}) and (\ref{eq:
eq10 teorem sa RUBOM}) that
\begin{align*}
S' &\subseteq
f(\tilde{A}_1) \cup f(\tilde{B}_1) \cup \ldots \cup f(\tilde{A}_n) \cup f(\tilde{B}_n) \cup f(E)\\
&\subseteq J_{a_1} \cup J_{b_1} \cup \ldots \cup J_{a_n} \cup
J_{b_n} \cup J_{e},
\end{align*}
hence
$$
S' \subseteq J_{a_1} \cup J_{b_1} \cup \ldots \cup J_{a_n} \cup
J_{b_n} \cup J_{e}.
$$

Let us summarize. If $l \in \mathbb{N}$ is such (\ref{Iisijece})
holds, then there exist $a_1, \ldots, a_n$, $b_1, \ldots, b_n$,
$e$, $t \in \mathbb{N}$ such that
\begin{enumerate}
\item[1)] $J_{a_i} \subseteq_{\mathcal{C}} J_{c_i} \mbox{ for
each } i \in \{ 1, \ldots, n \}$
\item[2)] $J_{b_i} \subseteq_{\mathcal{C}}
J_{d_i} \mbox{ for each } i \in \{1, \ldots, n-1 \}$
\item[3)] $J_e \subseteq_{\mathcal{C}} J_{\varphi(l)}$
\item[4)] $J_{a_i}\diamond_{\mathcal{D}}J_{b_i} \mbox{ for each } i \in \{ 1, \ldots, n \}$
\item[5)] $T' \subseteq J_t$
\item[6)] $J_{b_n}\diamond_{\mathcal{D}}J_t$
\item[7)] $S' \subseteq J_{a_1} \cup J_{b_1} \cup \ldots \cup
J_{a_n} \cup J_{b_n} \cup J_{e}$.
\end{enumerate}

Let $\Gamma$ be the set of all $(l,a_1, \ldots, a_n,b_1, \ldots,
b_n,e,t) \in \mathbb{N}^{2n+3}$ such that 1) - 7) hold.
Furthermore, let $\Omega$ be the set of all $l \in \mathbb{N}$ for
which there exist $a_1, \ldots, a_n$, $b_1, \ldots, b_n$, $e$, $t
\in \mathbb{N}$ such that $(l,a_1,\ldots,a_n$, $b_1,\ldots,b_n$,
$e$, $t) \in \Gamma$. We have proved the following: $$\mbox{if }l
\in \mathbb{N}\mbox{ is such that }
 I_l \cap f([-1,1]^{n-1} \times [0,1]) \neq \emptyset,\mbox{ then
 }l \in \Omega.$$

Suppose now that $l \in \Omega$. Let us prove that
\begin{equation}\label{IlsijeceS}
I_l \cap S \neq \emptyset.
\end{equation}
Since $l \in \Omega$, there exist $a_1, \ldots, a_n$, $b_1,
\ldots, b_n$, $e$, $t \in \mathbb{N}$ such that $(l,a_1, \ldots,
a_n$, $b_1, \ldots, b_n$, $e$, $t) \in \Gamma$. So, for the
numbers $l,a_1, \ldots, a_n$, $b_1, \ldots, b_n$, $e$, $t$ the
statements 1) - 7) hold. By  (\ref{eq: eq11 teorem sa RUBOM}) and
7) we have
$$
f([-2,2]^{n-1} \times [0,2]) \subseteq J_{a_1} \cup J_{b_1} \cup
\ldots \cup J_{a_n} \cup J_{b_n} \cup J_{e}
$$
and it follows
\begin{equation}
\label{eq: eq12 teorem sa RUBOM} [-2,2]^{n-1} \times [0,2]
\subseteq f^{-1}(J_{a_1}) \cup f^{-1}(J_{b_1})
\cup \ldots \cup f^{-1}(J_{a_n}) \cup
f^{-1}(J_{b_n})\cup f^{-1}(J_e).
\end{equation}

Let $i \in \{1, \ldots, n \}$. Then $J_{a_i} \subseteq J_{c_i}$ by
1) and it follows from (\ref{eq: eq2 teorem sa RUBOM}) and
(\ref{eq: eq4 teorem sa RUBOM}) that $J_{a_i} \cap f(B_i) =
\emptyset$. Therefore
\begin{equation}
\label{eq: eq13 teorem sa RUBOM} f^{-1}(J_{a_i}) \cap B_i
= \emptyset.
\end{equation}

Let $i \in \{1, \ldots, n-1 \}$. It follows from 2) and (\ref{eq:
eq3 teorem sa RUBOM}) that $J_{b_i} \cap f(A_i) = \emptyset$ which
gives
\begin{equation}
\label{eq: eq14 teorem sa RUBOM} f^{-1}(J_{b_i}) \cap A_i
= \emptyset.
\end{equation}
By (\ref{eq: eq15 teorem sa RUBOM}), 5) and 6) we have $J_{b_n}
\cap f(A_n) = \emptyset$. Thus
\begin{equation*}
f^{-1}(J_{b_n}) \cap A_n = \emptyset.
\end{equation*}
Statement 4) implies that
\begin{equation}
\label{eq: eq16 teorem sa RUBOM} f^{-1}(J_{a_i}) \cap
f^{-1}(J_{b_i}) = \emptyset
\end{equation}
for each $i \in \{ 1, \ldots, n \}$. The sets
$f^{-1}(J_{a_1}), \ldots, f^{-1}(J_{a_n}),
f^{-1}(J_{b_1}), \ldots, f^{-1}(J_{b_n})$ are
clearly open in $\mathbb{H}^n$. From (\ref{eq: eq13 teorem sa
RUBOM}), (\ref{eq: eq14 teorem sa RUBOM}), (\ref{eq: eq16 teorem
sa RUBOM}) and Lemma \ref{prop: A_i, A_n. B_i prije teoream sa
rubom} we conclude that
$$
[-2,2]^{n-1} \times [0,2] \nsubseteq f^{-1}(J_{a_1}) \cup
\ldots \cup f^{-1}(J_{a_n}) \cup f^{-1}(J_{b_1})
\cup \ldots \cup f^{-1}(J_{b_n}).
$$
From this and (\ref{eq: eq12 teorem sa RUBOM}) we get
$$
\big( [-2,2]^{n-1} \times [0,2] \big) \cap f^{-1}(J_e)
\neq \emptyset.
$$
\noindent
Hence $J_e \cap f([-2,2]^{n-1} \times [0,2]) \neq \emptyset$
which, together with  3) and $J_{\varphi(l)} = I_l$, gives
 $$I_l \cap f([-2,2]^{n-1} \times [0,2]) \neq \emptyset.$$
This clearly implies (\ref{IlsijeceS}).

We have proved that for each  $l \in \mathbb{N}$ the following
implications hold:
\begin{enumerate}
\item[$i$)] $I_l \cap f([-1,1]^{n-1} \times [0,1]) \neq \emptyset
\Rightarrow l \in \Omega$ \item[$ii$)] $l \in \Omega \Rightarrow
I_l \cap S \neq \emptyset$.
\end{enumerate}
It is easy to conclude that $\Omega $ is a c.e.\ set. So, by  $i$)
and $ii$), the set $f([ -1,1 ]^{n-1} \times [0,1])$ is c.e.\ up to
$S$. Clearly $f([ -1,1 ]^{n-1} \times [0,1])$ is a neighborhood of
$x$ in $S$.  Hence $S$ is c.e.\ at $x$.
\end{proof}

The following theorem is a generalization of Theorem
\ref{manifold-tm}.

\begin{theorem}\label{manifold-bound-tm}
\label{thm: S n-mnog s rubom i rub od S je poluizrac skup -- tada
je S izracunljiv skup} Let $(X,\mathcal{T}, (I_i))$ be a
computable topological space, let $n \in \mathbb{N}\setminus \{ 0
\}$ and let $S$ be a semicomputable set in $(X,\mathcal{T},(I_i))$
which has the following property: $S$ is, as a subspace of
$(X,\mathcal{T})$, an $n$--manifold with boundary and $\partial S$
is a semicomputable set in $(X, \mathcal{T}, (I_i))$. Then $S$ is
a computable set.
\end{theorem}

\begin{proof}
Since $S$ is compact, it suffices to prove that $S$ is locally
c.e.

Let $x \in S$. Then one of the following holds:
\begin{enumerate}
\item[1)] There exists a neighborhood of $x$ in $S$ which is
homeomorphic to $\mathbb{R}^n$. \item[2)] There exists a
neighborhood $N$ of $x$ in $S$ and a homeomorphism $f \colon
\mathbb{H}^n \to N$ such that
 $x \in f(\Bd \mathbb{H}^n)$.
\end{enumerate}
If 1) holds, then $S$ is c.e.\ at $x$ by Theorem \ref{lokalno-eukl}.

Suppose that 2) holds. We may assume that $N$ is an open
neighborhood of $x$. It is easy to conclude (see the proof of
Theorem 6.1 in \cite{lmcs:mnf}) that $$f(\Bd \mathbb{H}^n) = N
\cap
\partial S.$$
Now Theorem \ref{lokalno-eukl-rub} implies that $S$ is c.e.\ at
$x$.

So $S$ is locally c.e.\ and the claim of the theorem follows.
\end{proof}

\section{Compactification and semicomputability} \label{sect-cas}

If $(X,d)$ is a metric space, for $x\in X$ and $r>0$ by
$\hat{B}(x,r)$ we denote the closed ball in $(X,d)$ of radius $r$
centered in $x$, i.e.\ $\hat{B}(x,r)=\{y\in X\mid d(y,x)\leq r\}$.

Let $(X,d,\alpha )$ be a computable metric space. If $p\in
\mathbb{N}$ and $r$ is a positive rational number, then we say
that $\hat{B}(\alpha _{p} ,r)$ is a \textbf{rational closed ball}
in $(X,d,\alpha )$. For $i\in \mathbb{N}$ we define $$\hat{I}_{i}
=\hat{B}(\lambda _{i} ,\rho _{i} )$$ (recall the definition
(\ref{defIi})). Then $\{\hat{I}_{i}\mid i\in \mathbb{N}\}$ is the
family of all rational closed balls in $(X,d,\alpha )$.

Semicomputable (compact) sets in a computable metric space can be
characterized in the following way (see Proposition 3.1 in
\cite{lmcs:1mnf}).
\begin{proposition} \label{semi-semic.c.b}
Let $(X,d,\alpha )$ be a computable metric space and let $S$ be a
compact set in $(X,d)$. Then $S$ is semicomputable in $(X,d,\alpha
)$ if and only if $S\cap B$ is a compact set for each closed ball
$B$ in $(X,d)$ and the set $\{(i,j)\in \mathbb{N}^{2}\mid
\hat{I}_{i} \cap S\subseteq J_{j} \}$ is c.e.
\end{proposition}

Using this proposition, we extend the notion of a semicomputable
set in a computable metric space to noncompact sets.

Let $(X,d,\alpha )$ be a computable metric space and let $S$ be a
subset of $X$ (possibly noncompact). We say that $S$ is
\textbf{semicomputable} in $(X,d,\alpha )$ if the following holds
(see the definition of a semi-c.c.b.\ set in \cite{lmcs:1mnf}):
\begin{enumerate}
\item[(i)] $S\cap B$ is a compact set for each closed ball $B$ in
$(X,d)$; \item[(ii)] the set $\{(i,j)\in \mathbb{N}^{2}\mid
\hat{I}_{i} \cap S\subseteq J_{j} \}$ is c.e.
\end{enumerate}
For a compact set $S$ this definition, by Proposition
\ref{semi-semic.c.b}, coincides with the earlier definition of a
semicomputable set.

Condition (i) easily implies that each semicomputable set is
closed.

In view of equivalence (\ref{compset-ce-semi}) we extend the
notion of a computable set. If $(X,d,\alpha )$ is a computable
metric space and $S\subseteq X$, then we say that $S$ is
\textbf{computable} if $S$ is c.e.\ and semicomputable.

As before, we have that each computable set is a computable closed set (recall the definition of a computable closed set from Section \ref{sect-prelim}).
In computable metric spaces which have the effective covering property and compact closed balls, the notions ``computable set'' and ``computable closed set'' coincide \cite{lmcs:1mnf}.

Now it makes sense to ask does the implication
\begin{equation}\label{impl-S-nonc}
S\mbox{ semicomputable }\implies S\mbox{ computable}
\end{equation}
 hold for
noncompact manifolds $S$ (in a computable metric space)? In
general, the answer is negative. It is not hard to construct a
semicomputable $1$-manifold in $\mathbb{R}^{2}$ which is not
computable (see \cite{lmcs:1mnf}). On the other hand, if $S$ is a
$1$-manifold such that $S$ has finitely many connected components,
then (\ref{impl-S-nonc}) holds \cite{lmcs:1mnf}.

A general idea how to deal with the case when $S$ is noncompact
could be to apply certain construction which changes the ambient
space and which turns $S$ into a compact set (keeping the
semicomputability of $S$). This construction, which is similar to
a compactification of a space, leads to a new ambient space which
is not a metric space, but a topological space and this is where
the concept of a computable topological space will be applied.

Let us recall the notion of a one-point compactification. Suppose
$(X,\mathcal{T})$ is a topological space and $Y=X\cup \{\infty\}$,
where $\infty\notin X$. Let $$\mathcal{S} = \mathcal{T} \cup
\left\{\{\infty\} \cup U \mid U \in \mathcal{T} \mbox{ and } X
\setminus U \mbox{ is compact in } (X, \mathcal{T}) \right\}.$$
Then $(Y,\mathcal{S})$ is a compact topological space called a
one-point compactification of $(X,\mathcal{T})$.

Following the idea from this definition, we are going to use the
following construction. Suppose $(X,d,\alpha)$ is a computable
metric space and $Y = X \cup \{\infty\}$, where $\infty \notin X$.
Let
\begin{equation}\label{pseudo-top}
\mathcal{S} = \mathcal{T}_d \cup \left\{\{\infty\} \cup U \mid U
\mbox{ is open in } (X,d) \mbox{ and } X \setminus U \mbox{ is
bounded in } (X,d) \right\}.
\end{equation}
It is straightforward to check that $(Y,\mathcal{S})$ is a
topological space  and that $(X,\mathcal{T}_{d})$ is a subspace of
$(Y,\mathcal{S})$. For $i \in \mathbb{N}$ let
$$
B_i =
\begin{cases}
I_{\frac{i}{2}}, & \text{if } i \in 2\mathbb{N}\\
\{ \infty \} \cup \left( X \setminus \hat{I}_{\frac{i-1}{2}}
\right), & \text{if } i \in 2\mathbb{N} + 1.
\end{cases}
$$
We say that the triple $(Y,\mathcal{S},(B_{i} )_{i\in
\mathbb{N}})$ is a \textbf{pseudocompactification} of the
computable metric space $(X,d,\alpha)$.

We claim that $(Y,\mathcal{S},(B_{i} )_{i\in \mathbb{N}})$ is a
computable topological space. First, we have the following lemma.

\begin{lemma} \label{dodatno-F} Let $(X,d,\alpha)$
be a computable metric space.
\begin{enumerate}
\item[(i)] Let $x \in X$ and $i \in \mathbb{N}$ be such that $x
\notin \hat{I}_i$. Then there exists $j \in \mathbb{N}$ such that
$x \in I_j$ and $I_j\diamond I_i$. \item[(ii)] Let $i,j \in
\mathbb{N}$. Then there exists $k \in \mathbb{N}$ such that
$I_i\subseteq _{F}I_k$ and $I_j\subseteq _{F} I_k$.
\end{enumerate}
\end{lemma}

\begin{proof}
(i) Since $x \notin \hat{I}_i$, we have $\rho _{i}  < d(x,\lambda
_{i} )$. Choose a positive rational number $r$ such that $\rho
_{i} + 2r < d(x,\lambda _{i} )$ and  choose $k \in \mathbb{N}$ so
that
\begin{equation}
\label{dodatno-F-1} d(\alpha_k, x) < r.
\end{equation}
Then \begin{equation} \label{dodatno-F-2} d(\alpha_k, \lambda _{i}
)
> r + \rho _{i} .
\end{equation}
Indeed, if $d(\alpha_k, \lambda _{i} ) \leq r + \rho _{i} $, then
$$ d(x, \lambda _{i} ) \leq d(x,\alpha_k) + d(\alpha_k, \lambda
_{i} ) < r+r+\rho _{i}  = 2r + \rho _{i}  < d(x,\lambda _{i} ),$$
a contradiction.

Choose $l \in \mathbb{N}$ so that $(\alpha _{k},r) = (\lambda _{j}
,\rho _{j} )$. Then, by (\ref{dodatno-F-1}), we have $x \in I_j$
and, by (\ref{dodatno-F-2}), $I_j\diamond I_i$.

(ii) For any $n\in \mathbb{N}$ we can find  a positive rational
number $r$ such that
$$
d(\alpha_n, \lambda_{i} ) + \rho _{i} <r~\mbox{ and }~ d(\alpha_n,
\lambda _{j} ) + \rho _{j}  < r
$$
and then a number $k \in \mathbb{N}$ such that $(\alpha _{n},r) =
(\lambda _{k} , \rho _{k} )$ is the desired number.
\end{proof}

\begin{theorem}
Let $(Y,\mathcal{S},(B_{i} ))$ be a pseudocompactification of a
computable metric space $(X,d,\alpha )$. Then
$(Y,\mathcal{S},(B_{i} ))$ is a computable topological space.
\end{theorem}
\begin{proof}
Let $\mathcal{B}=\{B_{i} \mid i\in \mathbb{N}\}$. We first show
that $\mathcal{B}$ is a basis for the topology $\mathcal{S}$.
Clearly
$$
\mathcal{B} = \left\{ I_i \mid i \in \mathbb{N} \right\} \cup
\left\{ \{ \infty \} \cup \left( X \setminus \hat{I}_i \right)
\mid i \in \mathbb{N} \right\}
$$
and it is immediate that $\mathcal{B} \subseteq \mathcal{S}$.

Now we check that for each $V \in \mathcal{S}$ and each $x \in V$
there exists $B \in \mathcal{B}$ such that $x \in \mathcal{B}
\subseteq V$. Let $V \in \mathcal{S}$ and  $x \in V$. We have two
cases: $V \in \mathcal{T}_d$ and $V \notin \mathcal{T}_d$.

If $V \in \mathcal{T}_d$, then there exists  $i \in \mathbb{N}$
such that
  $x \in I_i \subseteq V$ and clearly $I_i \in \mathcal{B}$.

Suppose  $V \notin \mathcal{T}_d$. Then $V = \{ \infty \} \cup U$,
where $U$ is open in $(X,d)$ and
 $X \setminus U$ is bounded in $(X,d)$.
 We have $x \in \{ \infty \} \cup U$.
If $x \in U$, then there exists $i \in \mathbb{N}$ such that $x
\in I_i \subseteq U$ and so $x \in I_i \subseteq V$.

Suppose $x = \infty$. Certainly, there exists $i \in \mathbb{N}$
such that $X \setminus U \subseteq \hat{I}_i$, which implies $X
\setminus \hat{I}_i \subseteq U$ and we get
$$
\infty \in \{ \infty \} \cup \left( X \setminus \hat{I}_i \right)
\subseteq \{ \infty \} \cup U.
$$
Hence, there exists $B \in \mathcal{B}$ such that $\infty \in B
\subseteq V$. We conclude that $\mathcal{B}$ is a basis for
$\mathcal{S}$.

Let
\begin{align*}
\Gamma _{1} =&\left\{(i,j)\in \mathbb{N}^{2}\mid i,j \in 2\mathbb{N}\mbox{ and }I_{\frac{i}{2}}\diamond
I_{\frac{j}{2}}\right\},\\
\Gamma _{2} =&\left\{(i,j)\in \mathbb{N}^{2}\mid i \in 2\mathbb{N}, j \in 2\mathbb{N} + 1\mbox{ and }
I_{\frac{i}{2}}\subseteq _{F} I_{\frac{j-1}{2}}\right\},\\
\Gamma _{3} =&\left\{(i,j)\in \mathbb{N}^{2}\mid i \in 2\mathbb{N} + 1, j \in 2\mathbb{N}\mbox{ and }
I_{\frac{j}{2}}\subseteq _{F}I_{\frac{i-1}{2}}\right\},\\
\Gamma _{4} =&\left\{(i,j)\in \mathbb{N}^{2}\mid i,j \in 2\mathbb{N}\mbox{ and }I_{\frac{i}{2}}\subseteq _{F} I_{\frac{j}{2}}\right\},\\
\Gamma _{5} =&\left\{(i,j)\in \mathbb{N}^{2}\mid i \in 2\mathbb{N}, j \in 2\mathbb{N} + 1\mbox{ and }
I_{\frac{i}{2}}\diamond I_{\frac{j-1}{2}}\right\},\\
\Gamma _{6} =&\left\{(i,j)\in \mathbb{N}^{2}\mid i, j \in 2\mathbb{N} + 1\mbox{ and }I_{\frac{j-1}{2}}\subseteq _{F}I_{\frac{i-1}{2}}\right\}.
\end{align*}

Let
$$\mathcal{D}=\Gamma _{1} \cup \Gamma _{2} \cup \Gamma _{3} $$
and
$$\mathcal{C}=\Gamma _{4} \cup \Gamma _{5}\cup \Gamma _{6}.$$
We claim that $\mathcal{C}$ and $\mathcal{D}$
are characteristic relations for $(Y, \mathcal{S}, (B_i)_{i \in \mathbb{N}})$.

Using Proposition \ref{svojstva-diamond-F} we conclude that the sets
$\Gamma_{1},$ \dots , $\Gamma _{6}$ are c.e. So $\mathcal{C}$ and $\mathcal{D}$
are c.e. We now verify properties (1)-(4) from
the definition of a computable topological space.

\begin{enumerate}
\item[(1)]
Suppose $i,j \in \mathbb{N}$ are such that $(i,j) \in \mathcal{D}$.\\
\begin{enumerate}
\item[\underline{Case 1}] $(i,j) \in \Gamma_1$. Then
 $I_{\frac{i}{2}} \cap I_{\frac{j}{2}} = \emptyset$ and since $B_i = I_{\frac{i}{2}}$, $B_j = I_{\frac{j}{2}}$, we have
 $B_i \cap B_j = \emptyset$.
 \item[\underline{Case 2}] $(i,j) \in \Gamma_2$. Then
  $I_{\frac{i}{2}} \subseteq I_{\frac{j-1}{2}}$, which implies
  $I_{\frac{i}{2}} \subseteq \hat{I}_{\frac{j-1}{2}}$ and therefore
  $I_{\frac{i}{2}} \cap (\{ \infty \} \cup (X \setminus \hat{I}_{\frac{j-1}{2}}) ) = \emptyset$.
  So $B_i \cap B_j = \emptyset$.
\item[\underline{Case 3}] $(i,j) \in \Gamma_3$. In the same way we get $B_i \cap B_j = \emptyset$.\\
\end{enumerate}

\item[(2)]
Suppose $i,j \in \mathbb{N}$ are such that $(i,j) \in \mathcal{C}$.\\
\begin{enumerate}
\item[\underline{Case 1}] $(i,j) \in \Gamma_4$. Then
 $I_{\frac{i}{2}} \subseteq I_{\frac{j}{2}}$ and $B_i \subseteq B_j$.
 \item[\underline{Case 2}] $(i,j) \in \Gamma_5$. Then
 $I_{\frac{i}{2}} \cap \hat{I}_{\frac{j-1}{2}} = \emptyset$ and so
 $I_{\frac{i}{2}} \subseteq (X \setminus \hat{I}_{\frac{j-1}{2}}) \cup \{ \infty \}$.
 Hence, $B_i \subseteq B_j$.
 \item[\underline{Case 3}] $(i,j) \in \Gamma_6$. Then $\hat{I}_{\frac{j-1}{2}} \subseteq I_{\frac{i-1}{2}}$, which implies
 $\hat{I}_{\frac{j-1}{2}} \subseteq \hat{I}_{\frac{i-1}{2}}$
 and this gives $X \setminus \hat{I}_{\frac{i-1}{2}} \subseteq X \setminus \hat{I}_{\frac{j-1}{2}}$.
 So $B_i \subseteq B_j$.\\
\end{enumerate}

\item[(3)]
Suppose $x,y \in X \cup \{ \infty \}$, $x \neq y$.\\
\begin{enumerate}
\item[\underline{Case 1}] $x,y \in X$. Then there exist
$i,j \in \mathbb{N}$ such that $x \in I_i$, $y \in I_j$ and such that
$I_i\diamond I_j$. It follows $x \in B_{2i}$, $y \in B_{2j}$ and $(2i, 2j) \in \mathcal{D}$.
 \item[\underline{Case 2}] One of the points $x$ and $y$ is equal to
 $\infty$. We may assume $y= \infty$. Then clearly $x \in X$.
 Choose $j \in \mathbb{N}$ such that $x \in I_j$.
 Then there exists  $i \in \mathbb{N}$ such that $x \in I_i$ and
 $I_i\subseteq _{F}I_j$. It follows $x \in B_{2i}$, $\infty \in B_{2j+1}$
 and $(2i,2j+1) \in \mathcal{D}$.\\
\end{enumerate}

\item[(4)]
Suppose $i,j \in \mathbb{N}$ and $x \in B_i \cap B_j$.\\
\begin{enumerate}
\item[\underline{Case 1}] $i,j \in 2\mathbb{N}$. Then $x \in I_{\frac{i}{2}} \cap I_{\frac{j}{2}}$ and
therefore there exists $k \in \mathbb{N}$ such that $x \in I_k$
and $I_k\subseteq _{F}I_{\frac{i}{2}}$ and $I_k\subseteq _{F}I_{\frac{j}{2}}$.
So $x \in B_{2k}$, $(2k, i) \in \mathcal{C}$ and $(2k,j) \in \mathcal{C}$.
\item[\underline{Case 2}] $i \in 2\mathbb{N}$, $j \in 2\mathbb{N}+1$.
Then $x \in I_{\frac{i}{2}} \cap (\{ \infty \} \cup (X \setminus \hat{I}_{\frac{j-1}{2}}))$.
It follows $x \in I_{\frac{i}{2}}$ and $x \notin \hat{I}_{\frac{j-1}{2}}$.
By Lemma \ref{dodatno-F} there exists $l \in \mathbb{N}$ such that $x \in I_l$ and
$I_l\diamond I_{\frac{j-1}{2}}$.
We have $x \in I_{\frac{i}{2}} \cap I_l$ and therefore there exists
$k \in \mathbb{N}$ such that $x \in I_k$, $I_k\subseteq _{F}I_{\frac{i}{2}}$
and $I_k\subseteq _{F}I_l$. It follows  $I_k\diamond I_{\frac{j-1}{2}}$.
Hence we have $x \in B_{2k}$, $(2k, i) \in \mathcal{C}$ and $(2k,j) \in \mathcal{C}$.
\item[\underline{Case 3}] $i \in 2\mathbb{N} +1$, $j \in 2\mathbb{N}$. This is essentially Case 2.
\item[\underline{Case 4}] $i,j \in 2\mathbb{N} + 1$. Then
$$
x \in \left(\{ \infty \} \cup \left(X \setminus
\hat{I}_{\frac{i-1}{2}} \right) \right) \cap\left(\{ \infty \}
\cup \left(X \setminus \hat{I}_{\frac{j-1}{2}} \right) \right).
$$

\begin{itemize}
\item[\underline{Subcase 1}] $x \in X$. We have
$$x \notin \hat{I}_{\frac{i-1}{2}}~\mbox{ and }~x \notin
\hat{I}_{\frac{j-1}{2}}.$$ By Lemma \ref{dodatno-F} there exist
 $i', j' \in \mathbb{N}$ such that $x \in I_{i'}$, $I_{i'}\diamond I_{\frac{i-1}{2}}$,
  $x \in I_{j'}$ and
$I_{j'}\diamond I_{\frac{j-1}{2}}$.

We have $x \in I_{i'} \cap I_{j'}$ and so there exists $k \in \mathbb{N}$
such that $x \in I_k$, $I_k\subseteq _{F}I_{i'}$ and $I_k\subseteq _{F}I_{j'}$.
It follows $I_k\diamond I_{\frac{i-1}{2}}$ and $I_k\diamond I_{\frac{j-1}{2}}$.
We have $x \in B_{2k}$,
$(2k,i) \in \mathcal{C}$ and $(2k,j) \in \mathcal{C}$.

\item[\underline{Subcase 2}]  $x= \infty$. By Lemma \ref{dodatno-F}
 there exists $k \in \mathbb{N}$ such that $I_{\frac{i-1}{2}}\subseteq _{F}I_k$
 and $I_{\frac{j-1}{2}}\subseteq _{F}I_k$.
 We have $\infty \in B_{2k+1}$, $(2k+1, i) \in \mathcal{C}$ and
$(2k+1, j) \in \mathcal{C}$.
\end{itemize}
\end{enumerate}
\end{enumerate}

We have proved that $\mathcal{C}$ and $\mathcal{D}$ are characteristic relations
for  $(Y,\mathcal{S}, (B_i)_{i \in \mathbb{N}})$. Hence
$(Y,\mathcal{S}, (B_i)_{i \in \mathbb{N}})$ is a computable topological space.
\end{proof}

If a metric space $(X,d)$ has compact closed balls, then $(Y,\mathcal{S})$, where $\mathcal{S}$ is given by (\ref{pseudo-top}), is a one-point compactification of $(X,\mathcal{T}_{d})$. Moreover, we have the following proposition.
\begin{proposition} \label{CB-X-Y}
Let $(X,d)$ be a metric space, let $Y=X\cup \{\infty\}$, where $\infty\notin X$, and let $\mathcal{S}$ be given by (\ref{pseudo-top}). Suppose $K\subseteq X$ is such that $K\cap D$ is a compact
set in $(X,d)$ for each closed ball $D$ in $(X,d)$. Then $K\cup \{\infty\}$, as a subspace of $(Y,\mathcal{S})$, is a one-point compactification of $K$ (where $K$ is taken as a subspace of $(X,\mathcal{T}_{d})$).
In particular,  $K\cup \{\infty\}$ is a compact set in $(Y,\mathcal{S})$.
\end{proposition}

\begin{proof}
Let $V\subseteq K\cup \{\infty\}$. By the definition of the subspace topology, $V$ is open in $K\cup \{\infty\}$ if and only if there exists an open set $U$ in $(X,d)$ such that
$$V=K\cap U\mbox{ or }(V=(K\cap U)\cup \{\infty\}\mbox{ and }X\setminus U\mbox{ bounded in }(X,d)).$$

Suppose $U$ is open and $X\setminus U$ is bounded in $(X,d)$. Let $W=K\cap U$. Then $W$ is open in $K$ and $K\setminus W=K\setminus U$ is closed and bounded in $K$, which, together with the assumption of the proposition, gives  that  $K\setminus W$ is compact in $K$.

Conversely, if $W$ is an open set in $K$ such that $K\setminus W$ is compact in $K$, then $W=K\cap U$, where $U$ is open in $(X,d)$. Since $K$ is closed in $(X,d)$ (which follows from the assumption of the proposition), the set $U'=U\cup (X\setminus K)$ is open in $(X,d)$. We have $W=K\cap U'$ and $$X\setminus U'=(X\setminus U)\cap K=K\setminus U=K\setminus W,$$ hence $X\setminus U'$ is bounded in $(X,d)$.

Altogether, we have the following conclusion: $V$ is open in $K\cup \{\infty\}$ if and only if $V$ is open in $K$ or $V=W \cup \{\infty\}$, where $W$ is open in $K$ and $K\setminus W$ compact in $K$.
\end{proof}

Let $(X,d,\alpha)$ be a computable metric space. For $l \in \mathbb{N}$
we define $$L_l = \bigcap_{i\in [l]}\hat{I}_{i} .$$

Let $i,l \in \mathbb{N}$. We write $$I_{i} \diamond L_{l} $$
if there exists $j \in [l]$ such that $I_i\diamond I_j$.
Note: if $I_i\diamond L_l$, then  $I_i \cap L_l = \emptyset$.

Let $u,l \in \mathbb{N}$. We write $$J_u\diamond L_l$$
if $I_i\diamond L_l$ for each $i \in [u]$.
Note: if $J_u\diamond L_l$, then $J_u \cap L_l = \emptyset$.

The following proposition can be proved in the same fashion as Proposition
\ref{c.e.-CD}.
\begin{proposition} \label{Ju-Ll}
Let $(X,d,\alpha)$ be a computable metric space.
Then the sets  $$\{ (i,l) \in \mathbb{N}^2 \mid I_i\diamond L_l \}\mbox{
and }\{ (u,l) \in \mathbb{N}^2 \mid J_{u}\diamond L_l \}$$ are c.e.
\end{proposition}

\begin{lemma} \label{I-L-J}
Let $(X,d,\alpha)$ be a computable metric space.
\begin{enumerate}
\item[(i)]
Let $l \in \mathbb{N}$ and $x \in X$ be such that $x \notin L_l$.
Then there exists $i \in \mathbb{N}$ such that $x \in I_i$
and $I_i\diamond L_l$.
\item[(ii)]
Let $l \in \mathbb{N}$ and let $K$ be a nonempty compact set in $(X,d)$
such that $K \cap L_l = \emptyset$.
Then there exists $u \in \mathbb{N}$ such that $J_u\diamond L_l$ and
 $K \subseteq J_u$.
 \end{enumerate}
\end{lemma}
\begin{proof}
(i) Since $x \notin L_l$, there exists $j \in [l]$
such that $x \notin \hat{I}_{j}$.
By Lemma \ref{dodatno-F}
there exists $i \in \mathbb{N}$ such that $x \in I_i$
a $I_i\diamond I_j$ and it follows
$I_i\diamond L_l$.

(ii) Using (i) and the compactness of $K$ we conclude that there exist
$i_{0} ,\dots ,i_{n} \in \mathbb{N}$ such that
$$K\subseteq I_{i_{0} }\cup \dots \cup I_{i_{n} }$$ and
$I_{i_{0} }\diamond L_{l} $,\dots , $I_{i_{n} }\diamond L_{l} $. Now we take
$u\in \mathbb{N}$ such that $[u]=\{i_{0} ,\dots ,i_{n} \}$.
\end{proof}

\begin{proposition}
\label{S-L-J}
Let $(X,d,\alpha)$ be a computable metric space
and let $S$ be a semicomputable set in $(X,d,\alpha)$. Then the set
$$\{ (l,j) \in \mathbb{N}^2 \mid S \cap L_l \subseteq J_j \}$$
is c.e.
\end{proposition}
\begin{proof}
Let $l,j \in \mathbb{N}$. We claim that
\begin{equation}
\label{S-L-J-1}
S \cap L_l \subseteq J_j
\end{equation}
if and only if
\begin{equation}\label{S-L-J-2}
S \cap \hat{I}_{(l)_0} \subseteq J_j \text{~ or ~}(\exists u \in \mathbb{N} \mbox{ such that } J_u\diamond L_l \mbox{ and } S \cap \hat{I}_{(l)_0} \subseteq J_j \cup J_u)
\end{equation}
(recall the notation from Subsection \ref{subs-fr}).

Let us suppose that (\ref{S-L-J-1}) holds.
The set $S \cap \hat{I}_{(l)_0}$ is closed since it is compact. Therefore
$(S \cap \hat{I}_{(l)_0}) \setminus J_j$ is closed and, as a subset of a compact
set $S \cap \hat{I}_{(l)_0}$, it is also compact.

If $x \in (S \cap \hat{I}_{(l)_0}) \setminus J_j$, then
$x \in S$ and $x \notin J_j$, which, together with
(\ref{S-L-J-1}),
implies $x \notin L_l$. This means that
$((S \cap \hat{I}_{(l)_0})\setminus J_j) \cap L_l = \emptyset$.

If $(S \cap \hat{I}_{(l)_0})\setminus J_j = \emptyset$, then
obviously $S \cap \hat{I}_{(l)_0} \subseteq J_j$.

Suppose $(S \cap \hat{I}_{(l)_0})\setminus J_j \neq \emptyset$.
By Lemma \ref{I-L-J}
there exists $u \in \mathbb{N}$ such that $J_u\diamond L_l$
and $(S \cap \hat{I}_{(l)_0})\setminus J_j \subseteq J_u$.
It follows $S \cap \hat{I}_{(l)_0} \subseteq J_j \cup J_u$.

Hence, (\ref{S-L-J-1}) implies (\ref{S-L-J-2}).

Suppose now that (\ref{S-L-J-2}) holds.
If $S \cap \hat{I}_{(l)_0} \subseteq J_j$,
then from $L_l \subseteq \hat{I}_{(l)_0}$  it follows
$S \cap L_l \subseteq J_j$.
If there exists $u \in \mathbb{N}$
such that $J_u\diamond L_l$ and
$S \cap \hat{I}_{(l)_0}\subseteq J_j \cup J_u$, then
we have $J_u \cap L_l = \emptyset$ and it follows
$S \cap L_l \subseteq J_j$.

So the statements (\ref{S-L-J-1}) and (\ref{S-L-J-2}) are
equivalent.  Using Lemma \ref{S-minus-J}, Proposition \ref{Ju-Ll} and
the fact that $S$ is semicomputable it is easy to conclude that the set
of all $(l,j)\in \mathbb{N}^{2}$ for which (\ref{S-L-J-2}) holds is
c.e. This proves the claim of the proposition.
\end{proof}

The main idea about psudocompactifications is to reduce the problem of computability of noncompact semicomputable sets in $(X,d,\alpha )$ to  computability of (compact) semicomputable sets in  $(Y,\mathcal{S},(B_{i} ))$. Note the following: if the metric space $(X,d)$ is bounded, each semicomputable set in $(X,d,\alpha )$ is compact. Therefore, the case when $(X,d)$ is bounded is not interesting in view of pseudocompactifications.

\begin{proposition} \label{semi-semi-comp}
Let $(X,d, \alpha)$ be a computable metric space and let $(Y,\mathcal{S},(B_{i} ))$
be its pseudocompactification. Let $K$ be a semicomputable set in  $(X,d,\alpha)$. Suppose the metric space $(X,d)$ is unbounded.
\begin{enumerate}
\item[(i)] If $K$ is compact in $(X,d)$, then $K$ is semicomputable in  $(Y, \mathcal{S}, (B_i) )$.
\item[(ii)] If $K$ is not compact in $(X,d)$, then $K \cup \{ \infty \}$ is semicomputable in  $(Y, \mathcal{S}, (B_i) )$.
\end{enumerate}
\end{proposition}
\begin{proof}
 For $j \in \mathbb{N}$ let
$$C_j = B_{(j)_0} \cup \ldots \cup B_{(j)_{\overline{j}}}.$$
Let $\Phi, \Psi \colon \mathbb{N} \to \mathcal{P}(\mathbb{N})$ be defined by
$$\Phi(j)=[j] \cap 2\mathbb{N}\mbox{ and }\Psi(j)=[j] \cap (2\mathbb{N} + 1).$$
These functions are clearly c.f.v.
Let $j \in \mathbb{N}$. We have  $[j] = \Phi(j) \cup \Psi(j)$ and
\begin{equation}\label{semi-semi-comp-4}
C_j=\bigcup_{i \in [j]}B_i=
\bigcup_{i \in \Phi(j)}B_i \cup \bigcup_{i \in \Psi(j)}B_i=\bigcup_{i \in \Phi(j)}I_{\frac{i}{2}} \cup \bigcup_{i \in \Psi(j)}\left( \{ \infty \} \cup \left( X \setminus \hat{I}_{\frac{i-1}{2}} \right) \right).
\end{equation}

(ii) Suppose $K$ is not compact in $(X,d)$.  We want to prove that $K \cup \{ \infty \}$ is semicomputable in  $(Y, \mathcal{S}, (B_i) )$. By Proposition \ref{CB-X-Y} $K \cup \{ \infty \}$ is compact in $(Y,\mathcal{S})$, so it remains to prove  that
 the set
\begin{equation}\label{semi-semi-comp-5}
\{ j \in \mathbb{N} \mid K \cup \{ \infty \}  \subseteq C_j\}
\end{equation}
 is c.e.

 Let $j\in \mathbb{N}$. Using (\ref{semi-semi-comp-4}) we get
\begin{align*}
K \cup \{ \infty \} \subseteq C_j
&\Leftrightarrow
\Psi(j) \neq \emptyset \text{ and } K \subseteq \bigcup_{i \in \Phi(j)}I_{\frac{i}{2}} \cup \bigcup_{i \in \Psi(j)} \left( X \setminus \hat{I}_{\frac{i-1}{2}} \right)\\
&\Leftrightarrow
\Psi(j) \neq \emptyset \text{ and } K \subseteq \bigcup_{i \in \Phi(j)}I_{\frac{i}{2}} \cup \left( X \setminus \left( \bigcap_{i \in \Psi(j)} \hat{I}_{\frac{i-1}{2}} \right) \right).
\end{align*}
In general, if $A,B \subseteq X$, then $K \subseteq A \cup (X \setminus B)$ if and only if $K \cap B \subseteq A$. So
\begin{equation}
\label{eq: eq1 iz thm: (Y,S) pseudokompaktifikacija i.m.p (X,d,alpha) preko tocke infty. S poluizracunljiv Cb-kompaktan t.d. je S neomeden. Tada je S cap infty poluizracunlji skup}
K \cup \{ \infty \} \subseteq C_j
\Leftrightarrow
\Psi(j) \neq \emptyset \text{ and } K \cap \left( \bigcap_{i \in \Psi(j)} \hat{I}_{\frac{i-1}{2}} \right) \subseteq \bigcup_{i \in \Phi(j)}I_{\frac{i}{2}}.
\end{equation}
It is easy to conclude that the function $\Psi' \colon \mathbb{N} \to \mathcal{P}(\mathbb{N})$ defined by
$$\Psi'(j) = \left\{ \frac{i-1}{2} ~ \middle| ~ i \in \Psi(j) \right\}\mbox{ if }\Psi (j)\neq\emptyset, \mbox{ and }\Psi '(j)=\{0\}\mbox{ if }\Psi (j)=\emptyset $$ is c.f.v. Since the set $\{(j,l)\in \mathbb{N}^{2}\mid \Psi '(j)=[l]\}$ is computable (Proposition \ref{cfv}(3)) and for each $j\in \mathbb{N}$ there exists $l\in \mathbb{N}$ such that $\Psi '(j)=[l]$, there exists a computable function  $g \colon \mathbb{N} \to \mathbb{N}$ such that  $\Psi'(j) = [g(j)]$ for each $j \in \mathbb{N}$.

Let $j \in \mathbb{N}$ be such that $\Psi(j) \neq \emptyset$. Then
$$\bigcap_{i \in \Psi(j)} \hat{I}_{\frac{i-1}{2}}=\bigcap_{i \in \Psi'(j)} \hat{I}_{i}=\bigcap_{i \in [g(j)]} \hat{I}_{i}=L_{g(j)}.$$

By (\ref{eq: eq1 iz thm: (Y,S) pseudokompaktifikacija i.m.p (X,d,alpha) preko tocke infty. S poluizracunljiv Cb-kompaktan t.d. je S neomeden. Tada je S cap infty poluizracunlji skup}) for each  $j \in \mathbb{N}$ we have
\begin{equation}
\label{eq: eq2 iz thm: (Y,S) pseudokompaktifikacija i.m.p (X,d,alpha) preko tocke infty. S poluizracunljiv Cb-kompaktan t.d. je S neomeden. Tada je S cap infty poluizracunlji skup}
K \cup \{ \infty \} \subseteq C_j
\Leftrightarrow
\Psi(j) \neq \emptyset \text{ and } K \cap L_{g(j)} \subseteq \bigcup_{i \in \Phi(j)}I_{\frac{i}{2}}.
\end{equation}

The metric space $(X,d)$ is unbounded by the assumption of the proposition. It is easy to conclude that there exists a computable function  $\gamma \colon \mathbb{N} \to \mathbb{N}$ such that
$$I_i\diamond I_{\gamma(i)}$$
for each  $i \in \mathbb{N}$.

As above, we conclude that there exists a computable function $f \colon \mathbb{N} \to \mathbb{N}$ such that
\begin{equation}
\label{eq: eq3 iz thm: (Y,S) pseudokompaktifikacija i.m.p (X,d,alpha) preko tocke infty. S poluizracunljiv Cb-kompaktan t.d. je S neomeden. Tada je S cap infty poluizracunlji skup}
\bigcup_{i \in \Phi(j)}I_{\frac{i}{2}}
=
J_{f(j)}
\end{equation}
 for each $j\in \mathbb{N}$ such that $\Phi(j)\neq \emptyset$ and
$$
J_{f(j)} =I_{\gamma( (g(j))_0 )}
$$
for each $j\in \mathbb{N}$ such that $\Phi(j)\neq \emptyset$. In the second case, since
$I_{\gamma( (g(j))_0 )}\diamond I_{ (g(j))_0 }$, we have $I_{\gamma( (g(j))_0 )} \cap \hat{I}_{(g(j))_0 } = \emptyset$ and consequently
\begin{equation}\label{semi-semi-comp-2}
J_{f(j)} \cap L_{g(j)} = \emptyset.
\end{equation}

We claim that for each $j\in \mathbb{N}$ the following equivalence holds:
\begin{equation}
\label{semi-semi-comp-3}
K \cup \{ \infty \} \subseteq C_j
\Leftrightarrow
\Psi(j) \neq \emptyset \text{~ and ~} K \cap L_{g(j)} \subseteq J_{f(j)}.
\end{equation}

Suppose $j \in \mathbb{N}$ is such that $K \cup \{ \infty \} \subseteq C_j$. It follows from (\ref{eq: eq2 iz thm: (Y,S) pseudokompaktifikacija i.m.p (X,d,alpha) preko tocke infty. S poluizracunljiv Cb-kompaktan t.d. je S neomeden. Tada je S cap infty poluizracunlji skup}) that $\Psi(j) \neq \emptyset$ and $K \cap L_{g(j)} \subseteq \bigcup_{i \in \Phi(j)} I_{\frac{i}{2}}$.

If $\Phi(j) \neq \emptyset$, then, by (\ref{eq: eq3 iz thm: (Y,S) pseudokompaktifikacija i.m.p (X,d,alpha) preko tocke infty. S poluizracunljiv Cb-kompaktan t.d. je S neomeden. Tada je S cap infty poluizracunlji skup}), $K \cap L_{g(j)} \subseteq J_{f(j)}$. If $\Phi(j) = \emptyset$,  then $K \cap L_{g(j)} = \emptyset$  and  $K \cap L_{g(j)} \subseteq J_{f(j)}$.
In either case we have
\begin{equation}\label{semi-semi-comp-1}
\Psi(j) \neq \emptyset\mbox{ and }K \cap L_{g(j)}\subseteq J_{f(j)}.
\end{equation}

Conversely, suppose $j\in \mathbb{N}$ is such that (\ref{semi-semi-comp-1}) holds.

 If $\Phi(j) \neq \emptyset$, then $K \cap L_{g(j)} \subseteq \bigcup_{i \in \Phi(j)}I_{\frac{i}{2}}$ and it follows from (\ref{eq: eq2 iz thm: (Y,S) pseudokompaktifikacija i.m.p (X,d,alpha) preko tocke infty. S poluizracunljiv Cb-kompaktan t.d. je S neomeden. Tada je S cap infty poluizracunlji skup}) that $K \cup \{ \infty \} \subseteq C_j$.

If $\Phi(j) = \emptyset$, then, by (\ref{semi-semi-comp-2}),  $L_{g(j)} \cap J_{f(j)} = \emptyset$ which, together with $K \cap L_{g(j)} \subseteq J_{f(j)}$, gives $K \cap L_{g(j)} = \emptyset$.
So $K \cap L_{g(j)} \subseteq \bigcup_{i \in \Phi(j)}I_{\frac{i}{2}}$ and (\ref{eq: eq2 iz thm: (Y,S) pseudokompaktifikacija i.m.p (X,d,alpha) preko tocke infty. S poluizracunljiv Cb-kompaktan t.d. je S neomeden. Tada je S cap infty poluizracunlji skup}) implies $K \cup \{ \infty \} \subseteq C_j$.

So (\ref{semi-semi-comp-3}) holds.
It follows readily from Proposition \ref{S-L-J} and (\ref{semi-semi-comp-3}) that the set (\ref{semi-semi-comp-5}) is c.e.

(i) Suppose $K$ is compact in $(X,d)$. Since $(X,\mathcal{T}_{d})$ is a subspace of $(Y,\mathcal{S})$, we have that $K$ is compact in $(Y,\mathcal{S})$. To prove that the set
\begin{equation}\label{semi-semi-comp-6}
\{ j \in \mathbb{N} \mid K   \subseteq C_j\}
\end{equation}
 is c.e.,\ we proceed in a similar way as in (ii). First, for each $j\in \mathbb{N}$ we get
$$K \subseteq C_j \Leftrightarrow K \cap \left( \bigcap_{i \in \Psi(j)} \hat{I}_{\frac{i-1}{2}} \right) \subseteq \bigcup_{i \in \Phi(j)}I_{\frac{i}{2}},$$
where we take  $\bigcap_{i \in \Psi(j)} \hat{I}_{\frac{i-1}{2}}=X$ if $\Psi(j)=\emptyset $. Since $K$ is bounded in $(X,d)$, there exists $i_{0} \in \mathbb{N}$ such that $K\subseteq \hat{I}_{i_{0} }$. Let us take a computable function $g:\mathbb{N}\rightarrow \mathbb{N}$ such that
$$\bigcap_{i \in \Psi(j)} \hat{I}_{\frac{i-1}{2}}=L_{g(j)}$$
for each $j\in \mathbb{N}$ such that $\Psi(j)\neq\emptyset $ and
$$L_{g(j)}=\hat{I}_{i_{0} }$$
for each $j\in \mathbb{N}$ such that $\Psi(j)=\emptyset $. Then, for each $j\in \mathbb{N}$,
$$K  \subseteq C_j \Leftrightarrow  K \cap L_{g(j)} \subseteq \bigcup_{i \in \Phi(j)}I_{\frac{i}{2}}.$$ Now, in the same way as in (ii), we get that the set (\ref{semi-semi-comp-6}) is c.e. Thus $K$ is semicomputable in $(Y,\mathcal{S},(B_{i} ))$.
\end{proof}

\begin{proposition} \label{ce-ce-comp}
Let $(X,d, \alpha)$ be a computable metric space and let $(Y,\mathcal{S},(B_{i} ))$
be its pseudocompactification. Suppose  $K\subseteq X$ is such that $K \cup \{ \infty \}$ is a c.e.\ set in $(Y,\mathcal{S}, (B_i))$. Then $K$ is c.e.\ in $(X,d,\alpha)$.
\end{proposition}

\begin{proof}
Since $K \cup \{ \infty \}$ is closed in $(Y,\mathcal{S})$,  $(X,\mathcal{T}_d)$ is a subspace of $(Y,\mathcal{S})$ and $K=(K \cup \{ \infty \}) \cap X$, we have that $K$ is closed in $(X,d)$. The set
$$\Gamma=\left\{ i \in \mathbb{N} ~ \middle| ~ B_i \cap \left( K \cup \{ \infty \} \right) \neq \emptyset \right\}$$
is c.e.\ by the assumption of the proposition. Let $f \colon \mathbb{N} \to \mathbb{N}$, $f(i) = 2i$. For each $i\in \mathbb{N}$ we have
$$I_i \cap K \neq \emptyset\Leftrightarrow B_{2i} \cap \left(K \cup \{ \infty \} \right) \neq \emptyset\Leftrightarrow
2i \in \Gamma\Leftrightarrow
f(i) \in \Gamma\Leftrightarrow i \in f^{-1}(\Gamma).$$
 Thus $\{i\in \mathbb{N}\mid I_{i} \cap K\neq\emptyset \} = f^{-1}(\Gamma)$ and the claim follows.
\end{proof}

 As noted, the implication
\begin{equation}\label{noncom-man}
\partial K\mbox{ computable }\implies K\mbox{ computable }
\end{equation}
need not hold if $K$ is a noncompact semicomputable manifold with boundary. We are going to prove that (\ref{noncom-man}) holds in the special case when $K$ is homeomorphic to $\mathbb{R}^{n} $ or $\mathbb{H}^{n} $. Moreover, we will get that (\ref{noncom-man}) holds if a sufficiently large part of $K$ looks like $\mathbb{R}^{n} $ or $\mathbb{H}^{n} $. More precisely, we will observe a manifold $K$ for which there exists an open set $U\subseteq K$ such that $\overline{U}$ is compact and such that $K\setminus U$ is homeomorphic to $\mathbb{R}^{n} \setminus B(0,r)$ or $\mathbb{H}^{n} \setminus B(0,r)$, where $r>0$ and $B(0,r)$ is an open  ball in $\mathbb{R}^{n} $ with respect to the Euclidean metric. We may assume $r=1$ since  $\mathbb{R}^{n} \setminus B(0,r)\cong \mathbb{R}^{n} \setminus B(0,1)$ and $\mathbb{H}^{n} \setminus B(0,r)\cong \mathbb{H}^{n} \setminus B(0,1)$ (we use $X\cong Y$ to denote that topological spaces $X$ and $Y$ are homeomorphic). Furthermore, it is not hard to conclude that $\mathbb{H}^{n} \setminus B(0,1)\cong \mathbb{H}^{n}$.

For $n\geq 1$ let 
$$\mathbb{S}^{n-1} =\{x\in \mathbb{R}^{n} \mid \|x\|=1\},$$
where $\|\cdot \|$ is the Euclidean norm on $\mathbb{R}^{n} $.

If $X$ is a topological space and $A\subseteq X$, by $\overline{A}$ we denote the closure of $A$ in $X$.
\begin{lemma}\label{lema-prije-tm}
Let $n\geq 1$ and let $K$ be an $n$-manifold with boundary. Suppose that there exists an open set $U\subseteq K$ such that $\overline{U}$ is compact and $K\setminus U$ is homeomorphic to $\mathbb{R}^{n} \setminus B(0,1) $ or $\mathbb{H}^{n} $. Then the following holds.
\begin{enumerate}
\item[(i)] A one-point compactification $K\cup \{\infty\}$ of $K$ is an $n$-manifold with boundary; if $K\setminus U\cong\mathbb{R}^{n} \setminus B(0,1) $, the boundary of $K\cup \{\infty\}$ is $\partial K$,  and  if $K\setminus U\cong\mathbb{H}^{n}  $, the boundary of $K\cup \{\infty\}$ is $\partial K\cup \{\infty\}$.
\item[(ii)] If $K\setminus U\cong\mathbb{R}^{n} \setminus B(0,1) $, then $\partial K$ is compact. If $K\setminus U\cong\mathbb{H}^{n}  $, then $\partial K$ is not compact.
\end{enumerate}
\end{lemma}
\begin{proof}(i)
In general, if $X\cup \{\infty\}$ is a one-point compactification of a topological space $X$, then $X$ is clearly an open subspace of $X\cup \{\infty\}$. Therefore, if $x\in K$ and $N$ a neighborhood of $x$ in $K$, then $N$ is a neighborhood of $x$ in $K\cup \{\infty\}$. This means that we only have to prove that $\infty$ has a neighborhood in $K\cup \{\infty\}$ which is homeomorphic  either to $\mathbb{R}^{n} $ (if $K\setminus U\cong\mathbb{R}^{n} \setminus B(0,1) $) or $\mathbb{H}^{n} $ by a homeomorphism which maps $\infty$ to $\Bd \mathbb{H}^{n}$ (if $K\setminus U\cong\mathbb{H}^{n}  $).

Since $\overline{U}$ is compact, the set $(K\setminus \overline{U})\cup \{\infty\}$ is open in $K\cup \{\infty\}$ and obviously $(K\setminus \overline{U})\cup \{\infty\}\subseteq (K\setminus U)\cup \{\infty\}$. So $(K\setminus U)\cup \{\infty\}$ is a neighborhood of $\infty$ in $K\cup \{\infty\}$.

It is easy to verify the following general fact: if $Y$ is a closed subspace of a topological space $X$ and $Y\cup \{\infty\}$ and $X\cup \{\infty\}$ are one-point compactifications of $Y$ and $X$, then $Y\cup \{\infty\}$ is a subspace of $X\cup \{\infty\}$.

Therefore, the compactification $(K\setminus U)\cup \{\infty\}$ of $K\setminus U$ is a subspace of $K\cup \{\infty\}$. Using the fact that $(K\setminus U)\cup \{\infty\}$ is a neighborhood of $\infty$ in $K\cup \{\infty\}$, we conclude the following: if $N$ is a neighborhood of $\infty$ in $(K\setminus U)\cup \{\infty\}$, then $N$ is also a neighborhood of $\infty$ in $K\cup \{\infty\}$. So it suffices to find a neighborhood of $\infty$ in $(K\setminus U)\cup \{\infty\}$ with desired properties.

Let us suppose that $K\setminus U\cong\mathbb{R}^{n} \setminus B(0,1) $.
It suffices to prove that the point $\infty$ in the one-point compactification $(\mathbb{R}^{n} \setminus B(0,1))\cup \{\infty\}$ has a neighborhood homeomorphic to $\mathbb{R}^{n} $. But, as above, $(\mathbb{R}^{n} \setminus B(0,1))\cup \{\infty\}$ is a subspace of $\mathbb{R}^{n} \cup \{\infty\}$ and $(\mathbb{R}^{n} \setminus B(0,1))\cup \{\infty\}$ is a neighborhood of $\infty$ in $\mathbb{R}^{n} \cup \{\infty\}$. So it is enough to prove that $\infty$ has a neighborhood in in $\mathbb{R}^{n} \cup \{\infty\}$ which is homeomorphic to $\mathbb{R}^{n} $. However, this is clear since $\mathbb{R}^{n} \cup \{\infty\}$ is homeomorphic to $\mathbb{S}^{n} $ and $\mathbb{S}^{n} $ is an $n$-manifold.

Let us suppose now that $K\setminus U\cong\mathbb{H}^{n} $.
Since $\mathbb{H}^{n} \cong \mathbb{H}^{n}\cap B(0,1)$ by the homeomorphism $x\mapsto \frac{x}{1+\|x\|}$, we have
 $$K\setminus U\cong B(0,1)\cap \mathbb{H}^{n} .$$ So it is enough to prove that $\infty$ has a neighborhood in $(B(0,1)\cap \mathbb{H}^{n})\cup \{\infty\}$ which is homeomorphic to $\mathbb{H}^{n}$ by a homeomorphism which maps $\infty$ to $\Bd \mathbb{H}^{n}$.

Since the set $B(0,1)\cap \mathbb{H}^{n}$ is closed in $B(0,1)$, $(B(0,1)\cap \mathbb{H}^{n})\cup \{\infty\}$  is a subspace of $B(0,1)\cup \{\infty\}$. It is known that the function $f:B(0,1)\cup \{\infty\}\rightarrow \mathbb{S}^{n} $ given by $f(\infty)=(-1,0,\dots ,0)$, $f(0)=(1,0,\dots ,0)$ and
$$f(x)=\left(\cos\|x\|,\frac{x_{1} }{\|x\|}\sin\|x\|,\dots ,\frac{x_{n} }{\|x\|}\sin\|x\|\right)$$
for $x\in B(0,1)$, $x\neq 0$, $x=(x_{1} ,\dots ,x_{n} )$, is a homeomorphism. This function induces a homeomorphism
$$(B(0,1)\cap \mathbb{H}^{n})\cup \{\infty\}\rightarrow f((B(0,1)\cap \mathbb{H}^{n})\cup \{\infty\}).$$ However,
$$f((B(0,1)\cap \mathbb{H}^{n})\cup \{\infty\})=\mathbb{S}^{n} \cap \mathbb{H}^{n+1}$$
and $\mathbb{S}^{n} \cap \mathbb{H}^{n+1}$, i.e.\ upper half-sphere, is an $n$-manifold with boundary, its boundary is $\mathbb{S}^{n-1}\times \{0\}$. We conclude that $(B(0,1)\cap \mathbb{H}^{n})\cup \{\infty\}$ is an $n$-manifold with boundary and $\infty$ belongs to its boundary, meaning that  $\infty$ has a desired neighborhood in $(B(0,1)\cap \mathbb{H}^{n})\cup \{\infty\}$.

(ii) Suppose $f:\mathbb{R}^{n} \setminus B(0,1) \rightarrow K\setminus U$ is a homeomorphism. Then $K\setminus U$ is Hausdorff and since the set $f(\mathbb{S}^{n-1})$ is compact in $K\setminus U$, this set is closed in $K\setminus U$. But $K\setminus U$ is closed in $K$, so $f(\mathbb{S}^{n-1})$ is closed in $K$. It follows that the set
$$A=K\setminus (\overline{U}\cup f(\mathbb{S}^{n-1}))$$
is open in $K$. We have $A\subseteq K\setminus U$, so $A$ is open in $K\setminus U$ and it is therefore homeomorphic to the open subset $f(A)$ of $\mathbb{R}^{n} \setminus B(0,1)$. It follows from the definition of $A$ that $f(A)\subseteq \mathbb{R}^{n} \setminus \hat{B}(0,1)$ and, since the set $\mathbb{R}^{n} \setminus \hat{B}(0,1)$ is open in $\mathbb{R}^{n} $, $f(A)$ is open in $\mathbb{R}^{n} $. Hence $A$ is open subset of $K$ which is homeomorphic to an open subset of $\mathbb{R}^{n} $ and it follows that each point of $A$ has a neighborhood in $K$ homeomorphic to $\mathbb{R}^{n} $. So $A\cap \partial K=\emptyset $, hence $\partial K\subseteq \overline{U}\cup f(\mathbb{S}^{n-1})$. In general, the boundary of a manifold is a closed subset of the manifold. As a closed set contained in a compact set,  $\partial K$ is compact.

Suppose $f:\mathbb{H}^{n} \rightarrow K\setminus U$ is a homeomorphism. Since $(K\setminus U)\cap  \overline{U}$ is compact, the preimage by $f$ of this set is compact in $\mathbb{H}^{n}$ and we conclude  that there exists $r>0$ such that $f(\mathbb{H}^{n}\setminus \hat{B}(0,r))\subseteq K\setminus \overline{U}$. It follows that the set $f(\mathbb{H}^{n}\setminus \hat{B}(0,r))$ is open in $K$ and, consequently,
\begin{equation}\label{lema-top}
f(\Bd \mathbb{H}^{n}\setminus \hat{B}(0,r))\subseteq  \partial K.
\end{equation}

Suppose $\partial K$ is compact. The set $(K\setminus U)\cap \partial K$ is closed and contained in $\partial K$, hence it is compact. So $f^{-1} (\partial K)$ is a compact set in $\mathbb{H}^{n}$. But this contradicts (\ref{lema-top}). Thus $\partial K$ is not compact.
\end{proof}

\begin{theorem}\label{zadnji}
Let $(X,d,\alpha )$ be a computable metric space and let $K$ be a semicomputable set in this space which is, as a subspace of $(X,d)$, a manifold with boundary. Then the implication
 $$\partial K\mbox{ computable }\implies K\mbox{ computable }$$
 holds if there exists an open set $U$ in $K$ such that $\overline{U}$ is compact in $K$ and $K\setminus U$ is homeomorphic to $\mathbb{R}^{n} \setminus B(0,1) $ or $\mathbb{H}^{n} $.
\end{theorem}
\begin{proof}
We may assume that $K$ is noncompact, otherwise the claim follows from Theorem \ref{manifold-bound-tm} (or \cite{lmcs:mnf}). It follows that $(X,d)$ is unbounded.

Suppose that $\partial K$ is semicomputable in $(X,d,\alpha )$.
Let $(Y,\mathcal{S},(B_{i} ))$ be a pseudocompactification of $(X,d,\alpha )$.

Let us suppose that $K\setminus U\cong \mathbb{R}^{n} \setminus B(0,1)$ for some open set $U$ in $K$ such that $\overline{U}$ is compact. By Lemma \ref{lema-prije-tm}(ii) the set $\partial K$ is compact and, by Proposition \ref{semi-semi-comp}(i), $\partial K$ is semicomputable in $(Y,\mathcal{S},(B_{i} ))$. By  Proposition \ref{semi-semi-comp}(ii), $K\cup \{\infty\}$ is semicomputable in $(Y,\mathcal{S},(B_{i} ))$ and, by Lemma \ref{lema-prije-tm}(i) and Proposition \ref{CB-X-Y}, $K\cup \{\infty\}$ is a manifold with boundary and its boundary is $\partial K$. It follows from Theorem \ref{manifold-bound-tm} that $K\cup \{\infty\}$ is c.e.\ in $(Y,\mathcal{S},(B_{i} ))$. Proposition \ref{ce-ce-comp} now implies that $K$ is c.e.\ in $(X,d,\alpha )$. Hence $K$ is computable in $(X,d,\alpha )$.

Let us suppose that $K\setminus U\cong \mathbb{H}^{n}$ for some open set $U$ in $K$ such that $\overline{U}$ is compact. Using Proposition \ref{CB-X-Y}, Lemma \ref{lema-prije-tm} and Proposition \ref{semi-semi-comp} we get the following conclusion:
$K\cup \{\infty\}$ is a semicomputable manifold with boundary in $(Y,\mathcal{S},(B_{i} ))$, it boundary is $\partial K\cup \{\infty\}$ and $\partial K\cup \{\infty\}$ is a semicomputable set in $(Y,\mathcal{S},(B_{i} ))$.
Again, Theorem \ref{manifold-bound-tm} and Proposition \ref{ce-ce-comp} imply that $K$ is computable in $(X,d,\alpha )$.
\end{proof}

In particular, if $K$ is a semicomputable set in $(X,d,\alpha )$ such that $K\cong \mathbb{R}^{n} $, then $K$ is computable, and if $K$ is a semicomputable set for which there exists a homeomorphism  $f:\mathbb{H}^{n}\rightarrow  K$ such that $f(\Bd \mathbb{H}^{n})$ is a semicomputable set, then $K$ is computable.

\begin{example1}
Let $(X,d,\alpha )$ be a computable metric space and let $K$ be a semicomputable set in this space.

Suppose $K\cong \mathbb{S}^{1}\times [0,\infty\rangle$.  Then $K$ is a manifold with boundary and $\partial K=f(\mathbb{S}^{1}\times \{0\})$, where $f:\mathbb{S}^{1}\times [0,\infty\rangle\rightarrow K$ is a homeomorphism.
Suppose $\partial K$ is semicomputable set. Then $K$ is computable. This follows from Theorem \ref{zadnji} since $\mathbb{S}^{1}\times [0,\infty\rangle\cong \mathbb{R}^{2} \setminus B(0,1)$: the function $g:\mathbb{S}^{1}\times [0,\infty\rangle\rightarrow \mathbb{R}^{2} \setminus B(0,1)$, $g(x,t)=(1+t)x$, is a homeomorphism.

If we restrict $g$ to the product of the upper half-circle $\mathbb{S}^{1}\cap \mathbb{H}^{2}$ and $[0,\infty\rangle $, we get the conclusion that $[0,1]\times [0,\infty\rangle \cong \mathbb{H}^{2}\setminus B(0,1)$, hence
$[0,1]\times [0,\infty\rangle \cong \mathbb{H}^{2}$.

Suppose $K\cong [0,1]\times [0,\infty\rangle$. Then $K$ is a manifold with boundary and $\partial K=f([0,1]\times \{0\}\cup \{0,1\}\times [0,\infty\rangle )$, where $f:[0,1]\times [0,\infty\rangle\rightarrow K$ is a homeomorphism. By Theorem \ref{zadnji}, $K$ is computable if $\partial K$ is semicomputable.
\end{example1}

\section{Conclusion}
Semicomputable sets in Euclidean spaces (and in other usual spaces) naturally arise and it is of interest to know under which conditions these sets are computable. It is known that topology plays a important role in the description of such conditions. In particular, a semicomputable set is computable if it is a compact topological manifold (whose boundary is semicomputable).   In this paper we have shown that topology is actually involved in this matter at the basic level: the ambient space (Euclidean space or computable metric space) can be replaced by a computable topological space. Hence, to define necessary notions and to prove that semicomputable sets are computable under certain conditions, we do not need Euclidean space and we do not need even metric spaces: computable topological spaces are  sufficient.

Furthermore, it has been shown how the introduced concepts and results can be used to conclude that certain noncompact sets in computable metric spaces are computable. We believe that the subject of this paper has a potential for further investigations and applications.

\bibliographystyle{plain}
\bibliography{References}

\end{document}